\documentclass[11pt, leqno]{amsart}

\usepackage{graphicx}
\usepackage{verbatim}
\usepackage{color}
\usepackage{multicol}
\usepackage{stmaryrd}
\usepackage{parskip}
\usepackage{booktabs}
\usepackage{fancyhdr}

\usepackage[shortlabels]{enumitem}
\setlist[enumerate,1]{label=(\alph*), ref=(\alph*)}

\usepackage{hyperref}
\hypersetup{
  colorlinks=true,
  linkcolor=blue,
  urlcolor =blue,
  citecolor = magenta}

\usepackage{amssymb}
\usepackage{amsmath}
\usepackage{amsfonts}
\usepackage{amsthm}
\usepackage{epstopdf}
\usepackage{fullpage}

\newcommand{\MM}{{\mathcal M}}
\newcommand{\cQ}{{\mathcal Q}}

\newcommand{\OO}{{\mathcal O}}
\newcommand{\cond}{\mathsf{cond}}

\newcommand{\ZZ}{\mathbf{Z}}
\newcommand{\QQ}{\mathbf{Q}}

\renewcommand{\AA}{\mathbf{A}}
\newcommand{\cA}{\mathcal{A}}
\newcommand{\fabs}{F_{\text{abs}}}
\newcommand{\XX}{\mathbf{X}}
\newcommand{\HH}{\mathbf{H}}
\newcommand{\af}{\AA_f}
\newcommand{\qbar}{\overline{\QQ}}
\newcommand{\FF}{\mathbf{F}}
\newcommand{\fp}{\FF_p}
\newcommand{\fps}{\FF_{p^2}}
\newcommand{\fpbar}{\overline{\FF}_p}

\newcommand{\NN}{\mathbf{N}}
\newcommand{\RR}{\mathbf{R}}

\newcommand{\id}{\mathsf{id}}
\newcommand{\longto}{\longrightarrow}

\newcommand{\GL}{\mathsf{GL}}
\newcommand{\Mat}{\mathsf{M}}
\newcommand{\Gal}{\mathsf{Gal}}
\newcommand{\gq}{\Gal\big(\qbar/\QQ\big)}
\newcommand{\End}{\mathsf{End}}
\newcommand{\Hom}{\mathsf{Hom}}
\newcommand{\rank}{\mathsf{rank}}
\newcommand{\Spec}{\mathsf{Spec}\,}
\newcommand{\Ta}{\mathsf{Ta}}
\newcommand{\om}{\underline{\omega}}
\newcommand{\oh}{\mathcal{O}}

\newcommand{\st}{\,\colon\;}
\newcommand{\Nm}{\mathsf{Nm}}

\newcommand{\Br}{\mathsf{Br}}
\newcommand{\Cls}{\mathsf{Cls}}
\renewcommand{\ker}{\mathsf{ker}\,}
\newcommand{\op}{\mathrm{op}}

\newenvironment{psmallmatrix}
{\left(\begin{smallmatrix}}
    {\end{smallmatrix}\right)}

\usepackage[normalem]{ulem}

\usepackage{tikz-cd}
\usepackage{tikz}
\usetikzlibrary{arrows,positioning}
\usetikzlibrary{decorations.markings}
\usepackage{thmtools}

\numberwithin{equation}{section}

\theoremstyle{plain}

\newtheorem{thm}[equation]{Theorem}
\newtheorem*{thm*}{Theorem}
\newtheorem{prop}[thm]{Proposition}
\newtheorem{lemma}[thm]{Lemma}
\newtheorem{cor}[thm]{Corollary}
\newtheorem{notn}[thm]{Notation}
\newtheorem*{thma}{Theorem A}
\newtheorem*{thmb}{Theorem B}
\newtheorem*{corc}{Corollary C}

\theoremstyle{definition}
\newtheorem*{defnx}{Definition}
\newenvironment{defn}
{\pushQED{\qed}\defnx}
{\popQED\enddefnx}

\theoremstyle{remark}
\newtheorem{remarkx}[thm]{Remark}
\newenvironment{remark}
{\pushQED{\qed}\remarkx}
{\popQED\endremarkx}

\renewcommand{\det}{\mathsf{det}}
\renewcommand{\to}{\longto}

\renewcommand{\geq}{\geqslant}

\renewcommand{\leq}{\leqslant}

\newcommand{\ds}{\displaystyle}

\title{Hecke structure of quaternionic modular forms mod $p$}

\author{Yiannis Fam}
\address{London School of Geometry and Number Theory}
\email{yiannis.fam.23@ucl.ac.uk}

\author{Alexandru Ghitza}
\address{University of Melbourne}
\email{aghitza@alum.mit.edu}

\dedicatory{\`A Jean-Pierre Serre, avec gratitude}

\begin{document}

\begin{abstract}
  We relate the systems of Hecke eigenvalues arising from the (mod $p$) modular forms on the Shimura curve attached to the indefinite quaternion algebra $B$ of discriminant $\delta$ over $\QQ$ to the systems of Hecke eigenvalues arising from the (mod $p$) algebraic modular forms attached to the definite quaternion algebra $D$ of discriminant $p\delta$.

  Moreover, we discuss details of the Hecke structure on the latter spaces, following ideas of Serre.
  The entire setup can be seen as a Shimura curve analogue of \cite{serre-letters}, Serre's letter to Tate on quaternions and modular forms.

  The bijection between the sets of systems of Hecke eigenvalues is a special case of \cite{TY}; the novelty of this paper is the explicit nature of the construction, allowing for finer control of its behaviour with respect to weights, as well as the results on the Hecke module structure.
\end{abstract}

\maketitle

\section{Introduction}

In a 1987 letter \cite{serre-letters} to Tate, Serre described an approach to the study of the Hecke eigenvalues of modular forms (mod $p$).
This was strongly motivated by his recently published \cite{serre-duke} conjecture on the modularity of Galois representations (mod $p$), where he mused (question 2 on page 198) about the possible existence of a ``Langlands philosophy modulo $p$''.
Serre's method is to some extent a ``Jacquet--Langlands correspondence modulo $p$'': it relates the systems of eigenvalues of modular forms (mod $p$) to those of certain functions on the adelic group coming from the units in the definite quaternion algebra $D$ ramified at the prime $p$.
After proving this result, he commented on various benefits that this correspondence brings to the study of modular forms and the Galois representations attached to them; some of these benefits are sketched, while others are simply hinted at.

Following Serre's work, Gross developed a theory of algebraic modular forms \cite{gross-amf} on a reductive algebraic group $G$ over $\QQ$ satisfying a certain finiteness property, then conjectured \cite{gross-modp} that algebraic modular forms (mod $p$) would correspond to Galois representations valued in the Langlands dual group $\hat{G}$.
Viewed from this angle, Serre's letter
\begin{itemize}
  \item establishes a bijection between the systems of Hecke eigenvalues for modular forms (mod $p$) and those coming from algebraic modular forms (mod $p$) on the group $D^\times$, as long as forms of all weights are considered together;
  \item does so in an explicit manner, allowing one to determine when and how the weights change;
  \item identifies certain structural elements on the side of algebraic modular forms that simplify and illuminate the Hecke structure of modular forms (mod $p$).
\end{itemize}

The intervening years have seen successive generalisations of Serre's result (\cite{Ghitza1}, \cite{Reduzzi}), eventually encompassed by the work of Goldring--Koskivirta \cite{GK} and Terakado--Yu \cite{TY}.
In short, the main result of \cite{TY} establishes the bijection between the systems of Hecke eigenvalues for automorphic forms (mod $p$) on a Shimura variety of Hodge type for an algebraic group $G$, and those coming from algebraic modular forms (mod $p$) on an appropriate inner form of $G$.
A major ingredient of their proof is the generalisation in \cite[Section 11]{GK} of the geometric part of Serre's original argument, studying the restriction of automorphic forms (mod $p$) from the entire Shimura variety to a canonically defined zero-dimensional subset.

In the current paper, we establish Serre's result in the context of Shimura curves over $\QQ$.
From a purely quaternionic point of view, his ``Jacquet--Langlands correspondence mod $p$'' related the indefinite quaternion algebra $B=\Mat_2$ of discriminant $\delta=1$, to the definite quaternion algebra $D$ of discriminant $p$.
Our work relates an arbitrary indefinite quaternion algebra $B$ of discriminant $\delta>1$ to the definite quaternion algebra $D$ of discriminant $p\delta$; geometrically, our automorphic forms live on Shimura curves instead of modular curves.

\textbf{The following conventions will be in effect for the entire paper:} $p>3$ is a fixed prime number, $\delta>1$ and $N\geq 4$ are integers coprime to $p$ and to each other, and all Hecke operators we consider are indexed by primes $\ell$ that do not divide $p\delta N$.

We now describe our main results.
For the definition of the space of modular forms $M_k$ mod $p$ with respect to $B^\times$ of weight $k$ and level $V=V(N),V_1(N)$, see Section 2; this includes the Hasse invariant $\HH$ and the quotient $W_k=M_k/\HH M_{k-(p-1)}$.
For the definition of the double coset spaces $\Omega = \Omega^D(N), \Omega_1^D(N)$ of $D^\times(\af)$ see Section 3.
The $\fpbar$-valued functions on $\Omega$, denoted $\MM(\Omega)$, are the spaces of algebraic modular forms for $D^\times$ (mod $p$) of level $N$; $\MM_k(\Omega)$ are the weight $k$ subspaces.

\begin{thma}
  For every $k\geq 0$, restriction to the supersingular locus induces a Hecke-equivariant injective map $W_k\to\MM_k(\Omega)$.
  If $k>p+1$, this is an isomorphism.
\end{thma}

\begin{thmb}
  For every $k\in\ZZ$ there are explicit Hecke-equivariant isomorphisms
  \begin{equation*}
    \MM_k(\Omega)\cong \MM_{k+p^2-1}(\Omega),\quad
    \MM_k(\Omega)\cong \MM_{pk}(\Omega),\quad
    \MM_k(\Omega)[\chi]\cong \MM_{k+p+1}(\Omega),
  \end{equation*}
  where $\chi$ denotes the mod $p$ cyclotomic character.
\end{thmb}

\begin{corc}
  Let $(a_{\ell})\in\ds\prod_{\ell\nmid p\delta N} \fpbar$ be a system of Hecke eigenvalues.
  \begin{enumerate}
    \item If $(a_\ell)$ arises from $M_k$, then it arises from $\MM_{k'}(\Omega)$ for some $0\leq k'\leq k$ such that $k'\equiv k\pmod{p-1}$.
    \item If $(a_\ell)$ arises from $\MM_k(\Omega)$, then it arises from $M_{k'}$ for every $k'>p+1$ such that $k'\equiv k,pk\pmod{p^2-1}$.
  \end{enumerate}
  In particular, the systems of Hecke eigenvalues arising from $\ds\bigoplus_k M_k$ and from $\MM(\Omega)$ coincide.
\end{corc}

% \begin{thmmain}
%   The systems of Hecke eigenvalues $(a_\ell)$ arising from the modular forms mod $p$ with respect to $B^\times$, over all weights $k$ and fixed level $V$, are the same as those arising from $\MM(\Omega)$.
% \end{thmmain}

The last statement of the corollary is a special case of the results of Terakado--Yu we already mentioned, see \cite[Theorem 0.1 and Section 6.1]{TY}.
% The additional contribution that our paper makes is threefold:
% \begin{itemize}
%   \item the correspondence is obtained in a constructive, hands-on way, that requires a similar level of algebraic-geometric machinery as Serre's original work;
%   \item this gives us better control over the weights involved on the two sides of the bijection than what Terakado--Yu can do in their much more general setting;
%   \item we study the structure of the Hecke modules appearing on the algebraic modular forms side, establishing or generalising several of Serre's observations.
% \end{itemize}

The paper is organised as follows. In Section 2, we define the Shimura curve attached to the indefinite quaternion algebra $B$ and recall the theory of automorphic forms (mod $p$) defined on this Shimura curve. We recall the definition of the Hasse invariant with respect to $B^\times$ that played a central role in Serre's argument. In particular, we prove Proposition \ref{Hasse quotient 2}, which generalises the fact that multiplication by the classical Hasse invariant preserves mod $p$ Hecke eigenvalues. In Section 3 we recall the Serre--Gross algebraic modular forms (mod $p$) on the definite quaternion algebra $D$ and study their structure as Hecke modules, proving Theorem B. In Section 4, we relate the zero locus of the Hasse invariant on the Shimura curve to the adelic double coset spaces of Section 3 by a generalisation of the Deuring correspondence for supersingular elliptic curves. This allows us, in Section 5, to explicitly relate the systems of Hecke eigenvalues of automorphic forms for $B^\times$ (mod $p$) and algebraic modular forms for $D^\times$ (mod $p$), adapting the argument of Serre. This results in proofs of Theorem A and Corollary C.

\textbf{Acknowledgements:}
The authors thank David Kohel and Ariel Pacetti for their comments and corrections.
This paper originated in the first author's MPhil thesis at the University of Melbourne, which was supported by the University of Melbourne Robert George Williams Scholarship, as well as the Australian Government Research
Training Program Scholarship. The first author is currently supported by the Engineering and Physical Sciences Research Council [EP/S021590/1], the EPSRC Centre for Doctoral Training
in Geometry and Number Theory (The London School of Geometry and Number Theory) at Imperial College
London.

\section{Automorphic forms for $B^\times$ (mod $p$)}

Let $B$ be the unique indefinite non-split quaternion algebra over $\QQ$ with discriminant $\delta>1$ coprime to $p$.
The splitting of $B$ at $\infty$, so that $B^\times(\RR)\cong\GL_2(\RR)$, allows one to construct a Shimura curve $\XX$, analogous to the modular curves in the $\GL_2 / \QQ$ setting. We may then define modular forms (mod $p$) on this Shimura curve, with associated Hecke operators, and study the systems of Hecke eigenvalues.
%This contribution from the archimedean place was not apparent in the theory of locally constant functions on $D^\times$ so far, and this is a consequence of the fact that $D^\times(\RR)$ (the units of Hamilton's quaternions) is compact modulo the centre. Such a condition trivialises many of the conditions defining general automorphic forms.

As with the modular curve, the Shimura curve represents a moduli problem that we now describe. See also \cite{Mil1}, \cite{DT1}, \cite{Buz2}, \cite{Kas1} or \cite{Cla1}. Fix a maximal order $\OO_B$ of $B$ and fix isomorphisms $\OO_B \otimes \ZZ_\ell \cong \Mat_2(\ZZ_\ell)$ for all primes $\ell \nmid \delta$, and $\OO_B\otimes \ZZ/N\ZZ\cong \Mat_2(\ZZ/N\ZZ)$ for any $N$ coprime to $p\delta$.

\begin{defn}
  A false elliptic curve over a scheme $S$ (on which $\delta$ is invertible) is an abelian surface $A$ over $S$, together with an injective ring homomorphism $\iota\st \OO_B \hookrightarrow \End_S(A)$. An isogeny $f\st A \to A_0$ of false elliptic curves is an isogeny of abelian surfaces over $S$ that commutes with the actions of $\OO_B$; in other words we have the commutative diagram
  % $$\xymatrix{
  %     A \ar[r]^f \ar[d]_{\iota(b)} & A_0 \ar[d]^{\iota_0(b)} \\ A \ar[r]_f & A_0}$$
  \begin{center}
    \begin{tikzcd}
      A \arrow[r, "f"] \arrow[d, "\iota(b)"'] & A_0 \arrow[d, "\iota_0(b)"] \\
      A \arrow[r, "f"'] & A_0
    \end{tikzcd}
  \end{center}
  for every $b \in \OO_B$.
\end{defn}

\begin{remark}
  The term ``fausse courbe elliptique'' is attributed to Serre by Deligne and Rapoport \cite[p. 155]{DR}.
  David Kohel points out that translating this as ``fake elliptic curve'' does a better job of capturing its essence.
  We decided to stick with ``false elliptic curve'' as it is employed by all the references we are using.
\end{remark}

As in \cite[Remark 1.3]{Mil1}, Shimura varieties typically parametrise certain abelian varieties together with a polarisation $\lambda\st A \to A^\vee$. Such a polarisation induces the Rosati involution on $\End(A) \otimes \QQ$, which is known to be positive for the trace form. On the other hand, one can define involutions on $B$, and one wants this to coincide with the Rosati involution on $B \subset \End(A)\otimes\QQ$. Such an involution must necessarily be positive. Let $x \mapsto \bar{x}$ be the canonical involution on $B$. One can find $t \in \OO_B$ such that $t^2=-\delta$, then the involution $x^* := t^{-1}\bar{x} t$ is a positive involution on $\OO_B$. There exists a unique principal polarisation of $A$ such that the corresponding Rosati involution induces $*$ on $\OO_B$. Consequently, by fixing $t$, any false elliptic curve admits a canonical principal polarisation.

Given these principal polarisations on false elliptic curves, any isogeny of false elliptic curves $f\st A \to A_0$ has a dual isogeny $f^\vee\st A_0^\vee \to A^\vee$, which when composed with the principal polarisations gives an isogeny of false elliptic curves $f^t\st A_0 \to A$,
% $$\xymatrix{
%   A \ar[r]^f \ar[d]^\lambda & A_0 \ar[d]^{\lambda_0} \\
%   A^\vee \ar@/^/[u] & A_0^\vee \ar@/^/[u] \ar[l]^{f^\vee}}$$
\begin{center}
  \begin{tikzcd}
    A \arrow[r, "f"] \arrow[d, "\lambda"] & A_0 \arrow[d, "\lambda_0"] \\
    A^\vee \arrow[u, bend left] & A_0^\vee \arrow[u, bend left] \arrow[l, "f^\vee"]
  \end{tikzcd}
\end{center}
where $f^t \circ f\st A \to A$ is locally given by multiplication by an integer. If this integer is constant on $S$ it is called the false degree of $f$. For instance, the false degree of $[n]$ is $n^2$.

We define level structures on false elliptic curves in analogy to those defined on elliptic curves.

\begin{defn}
  Define level $V(N)$ structure on a false elliptic curve $A/S$ to be an isomorphism
  \begin{equation*}
    \alpha\st (\OO_B \otimes \ZZ/N\ZZ)_S \cong A[N]
  \end{equation*}
  of group schemes over $S$ with left action of $\OO_B$.
  Define level $V_1(N)$ structure to be an inclusion
  \begin{equation*}
    \alpha\st (\ZZ/N\ZZ \times \ZZ/N\ZZ)_S  \hookrightarrow A[N]
  \end{equation*}
  of group schemes over $S$ with left action of $\OO_B$  (defined on $\ZZ/N\ZZ \times \ZZ/N\ZZ$ by the identification \\ $\OO_B \otimes \ZZ/N\ZZ \cong \Mat_2(\ZZ/N\ZZ)$).
\end{defn}

We will also need the fact that the quotient $A/G$ of a false elliptic curve $A$ by a finite flat subgroup scheme $G$ that is stable under $\OO_B$ is naturally a false elliptic curve with action of $\OO_B$ induced by $\iota$. We have an isogeny $\pi\st A \to A/G$. If we have an isogeny $\eta\st A \to A_0$ of false elliptic curves, then for $H= (\OO_B \otimes \ZZ/N\ZZ)_S$ or $(\ZZ/N\ZZ \times \ZZ/N\ZZ)_S$, and $\alpha\st H \to A[N]$ a level structure, we define $\alpha_\eta = \eta \circ \alpha\st H \to A_0[N]$. This is a level structure when $\eta$ has false degree coprime to $N$.

\begin{thm}
  For $N \geq4$, the functor from the category of schemes over $\ZZ[1/N\delta]$ to the category of sets, assigning to any scheme $S$ the set of isomorphism classes of false elliptic curves over $S$ with level $V(N)$ structure, is representable by a smooth proper curve $\XX(N)$ over $\ZZ[1/N\delta]$. This is the Shimura curve of level $V(N)$ (associated to $B$). The same is true replacing $V(N)$ with $V_1(N)$, and we denote the Shimura curve by $\XX_1(N)$.
\end{thm}
\begin{proof}
  See \cite{DT1} or \cite{Buz2}.
\end{proof}
\begin{remark}
  In contrast to the theory of modular curves, one does not need to add cusps in the construction of the Shimura curves.
\end{remark}

We will denote again by $\XX = \XX(N), \XX_1(N)$ the base change over $\fpbar$. Over $\XX$ lies a universal false elliptic curve $\pi\st \cA \to \XX$ with level $V(N),V_1(N)$ structure. This is universal in the sense that for any false elliptic curve $(A,\iota,\alpha)$ with level $V(N),V_1(N)$ structure $\alpha$ over a scheme $S$ (over $\fpbar$), there exists a unique morphism $g\st S \to \XX$ such that $(A,\iota,\alpha)$ is the pullback of $\cA$ with its level structure. By this we mean that $A = \cA \times_{\XX} S$ with $\iota = \iota_{\mathrm{univ}} \times \id_S\st \OO_B \hookrightarrow \End(\cA \times_{\XX} S)$, and for $H= (\OO_B \otimes \ZZ/N\ZZ)$ or $(\ZZ/N\ZZ \times \ZZ/N\ZZ)$ as appropriate, the level structure $\alpha$ on $(A,\iota)$ is defined by the universal property of the fibre product $A= \cA \times_{\XX} S$ in the diagram below:

% $$\xymatrix{
%     H_{S} \ar[rr] \ar[rd]^{\alpha} \ar[dd]& & H_{\XX} \ar[rd]^{\alpha_{\mathrm{univ}}} \ar[dd] & \\
%     &  \cA \times_{\XX} S \ar[rr] \ar[ld]^s & & \cA \ar[ld]^\pi \\
%     S \ar[rr]_g & & \XX &}$$

\begin{center}
  \begin{tikzcd}[column sep=small, row sep=large]
    H_{S} \arrow[rr] \arrow[rd, "\alpha"] \arrow[dd] & & H_{\XX} \arrow[rd, "\alpha_{\mathrm{univ}}"] \arrow[dd] & \\
    & \cA \times_{\XX} S \arrow[rr, crossing over] \arrow[ld, "s"] & & \cA \arrow[ld, "\pi"] \\
    S \arrow[rr, "g"'] & & \XX &
  \end{tikzcd}
\end{center}

We wish to define modular forms on Shimura curves as global sections of (a power of) some line bundle $\om$ on $\XX$, or alternatively as rules assigning to $(A,\iota,\alpha,\omega)$ some value in $\fpbar$, where $\omega$ is a nonzero holomorphic differential on $A$.

%As false elliptic curves are abelian surfaces, in order to produce a line bundle we must exploit the additional structure coming from the action of $\OO_B$.

Let $s\st A \to S$ be a false elliptic curve. We have an action of $\OO_B$ on the topological space of $A$, and any $x \in \OO_B$, by definition, gives a morphism of sheaves $x^\#\st \OO_A \to x_* \OO_A$, and similarly $x^\#\st \Omega_{A/S}^1 \to x_*\Omega_{A/S}^1$. Since $\OO_B$ are endomorphisms of $A/S$, there is an induced morphism of sheaves on $S$,
\begin{equation*}
  s_* \Omega_{A/S}^1 \to s_*x_*\Omega_{A/S}^1 = s_*\Omega_{A/S}^1.
\end{equation*}
So $\OO_B^{\op}$ acts on the pushforward $s_*\Omega_{A/S}^1$. Tensoring with $\fp$, we have an action of $\Mat_2(\fp)$ on $s_*\Omega_{A/S}^1$. The point is that $\Mat_2(\fp)$ contains nonzero idempotents $e_1, e_2=1-e_1$, for example $\begin{psmallmatrix}1&0\\0&0\end{psmallmatrix}$ and $\begin{psmallmatrix}0&0\\0&1\end{psmallmatrix}$, that are conjugate. This allows us to take the 1-dimensional submodule $e_2 \cdot s_*\Omega_{A/S}^1$ of $s_*\Omega_{A/S}^1$ without losing any information.

\begin{notn}
  For our fixed choice of nonzero idempotents $e_1, e_2=1-e_1$ in $\Mat_2(\fp)$, define for any false elliptic curve $s\st A \to S$ the line bundle on $S$
  \begin{equation*}
    \om_{A/S} := e_2 \cdot s_* \Omega_{A/S}^1.
  \end{equation*}
  In the case of the universal false elliptic curve $\pi\st \cA\to\XX$, we write simply $\om$ for $\om_{\cA/\XX}$.
  Note that for a general $A\to S$, $\om_{A/S}$ is the pullback under the associated morphism $S \to \XX$ of $\om$.
\end{notn}

\begin{thm}\label{KS2}
  We have the Kodaira--Spencer isomorphism $\Omega_{\XX}^1 \cong \om^{\otimes 2}$.
\end{thm}
\begin{proof}
  \cite[Lemma 7]{DT1}.
  %Compare the result with the Kodaira-Spencer isomorphism for modular curves - on the Shimura curve there are no cusps.
\end{proof}

Having described the moduli interpretation of Shimura curves, we may now define modular forms on $\XX = \XX(N), \XX_1(N)$. As there are no cusps on Shimura curves, we do not require a condition of holomorphy at the cusps. For the same reason, there is no notion of the $q$-expansion of a modular form on a Shimura curve.

\begin{defn}
  A modular form with respect to $B^\times$, of weight $k \geq 0$ and level $V=V(N),V_1(N)$ over a ring $R_0$ on which $\delta N$ is invertible, is a rule assigning a value $f(A,\iota,\alpha,\omega) \in R$ to each isomorphism class of quadruples $(A,\iota,\alpha,\omega)$, where $R$ is an $R_0$-algebra, $(A,\iota,\alpha)$ is a false elliptic curve over $R$ with level $V$ structure, and $\omega$ is a basis for $\om_{A/R}$, satisfying:
  \begin{enumerate}[label=(\roman*)]
    \item $f(A/R,\iota,\alpha,\mu\omega) = \mu^{-k}f(A/R,\iota,\alpha,\omega)$ for any $\mu \in R^\times$.
    \item (Base change) If $\phi\st R \to R'$ is an $R_0$-algebra homomorphism, then $f(A/R',\iota,\alpha ,\omega_{R'}) = \phi(f(A/R,\iota,\alpha,\omega))$.
  \end{enumerate}
  Equivalently, $f$ is a global section of $H^0(\XX_{R_0} , \om^{\otimes k})$.
\end{defn}

From now on, we will consider the case $R_0= \fpbar$, and denote the space of these mod $p$ modular forms with respect to $B^\times$, of weight $k$ and level $V=V(N),V_1(N)$, by $M_k$ (for convenience we will omit the level $V$ and the quaternion algebra $B$).

We also define Hecke operators on modular forms on Shimura curves. For classical modular forms, as defined in \cite{Kat1}, the $\ell$-th Hecke operator is defined by indexing over degree $\ell$ isogenies $E \to E/C$, where $C$ is a subgroup scheme isomorphic to $\ZZ/\ell \ZZ$. For false elliptic curves $A$, we may only quotient by finite flat subgroup schemes stable under the action of $\OO_B$.

%Hence the natural analogue of Hecke operators for modular forms on Shimura curves should index over finite flat subgroup schemes $G \leq A$ isomorphic to $\ZZ/\ell\ZZ \times \ZZ/\ell\ZZ$ which are stable under the action of $\OO_B \otimes \ZZ/\ell \ZZ \cong \mathrm{M}_2(\ZZ/\ell \ZZ)$. Given that there exists an isomorphism $A[\ell] \cong \mathrm{M}_2(\ZZ/\ell\ZZ)$ as $\OO_B \otimes \ZZ/\ell \ZZ \cong \mathrm{M}_2(\ZZ/\ell\ZZ)$-modules, it is not hard to check that there are $\ell+1$ subgroups of $A[\ell]$ isomorphic to $\ZZ/\ell\ZZ \times \ZZ/\ell\ZZ$ that are stable under $\mathrm{M}_2(\ZZ/\ell\ZZ)$. These are generated under the action of $\mathrm{M}_2(\ZZ/\ell \ZZ)$ by the matrices $\begin{psmallmatrix}1&i \\0&0 \end{psmallmatrix}$ and $\begin{psmallmatrix}0&1 \\0&0 \end{psmallmatrix}$ for $0 \leq i \leq \ell-1$.

\begin{defn}\label{Hecke def 2}
  For a prime $\ell$ (recall, not dividing $p\delta N$), we define the $\ell$-th Hecke operator $T_\ell$ on $M_k$ by the following formula:
  \begin{equation*}T_\ell f (A,\iota,\alpha,\omega) = \frac{1}{\ell}\cdot \sum\limits_{\parbox{8em}{\scriptsize$\eta\st A \to A/G$ \\ $G \cong \ZZ/\ell\ZZ \times \ZZ/\ell\ZZ$ \\ as $\OO_B \otimes \ZZ/\ell\ZZ$-modules }} f(A/G, \iota, \alpha_\eta, (\eta^t)^*\omega).
  \end{equation*}
  % I've changed \ell^k-1 to \ell^-1 in line with the definition for \mathcal M_k
  Here $\eta^t$ was earlier defined as the composition of the dual isogeny $\eta^\vee$ with the principal polarisations, and we set $\alpha_{\eta}=\eta \circ \alpha$, which is a level structure on $A/G$ because $\ell$ is coprime to $N$. Note that $(\eta^t)^* \omega \in \om_{A/G}$ because $\eta^t$ commutes with $\OO_B$, and in particular with the idempotents $e_1,e_2$.
\end{defn}

We now generalise the Hasse invariant that plays a central role in the classical theory of mod $p$ modular forms. Let $(A,\iota)$ be a false elliptic curve over an $\fpbar$-algebra $R$, and let $\omega$ be a basis of $\om_{A/R}$. By \cite[Lemma 7]{DT1}, $\omega$ determines a dual basis $\omega^\vee \in e_2 \cdot H^1(A,\OO_A)$. Consider the absolute Frobenius morphism $\fabs\st A \to A$, which is the identity on points, and on structure sheaves is given by
\begin{align*}
  \fabs\st \OO_A & \to  \OO_A   \\
  f              & \mapsto f^p.
\end{align*}
This induces an $\fp$-linear map $\fabs^*\st H^1(A,\OO_A) \to H^1(A,\OO_A)$.

\begin{defn}
  Define the {\it Hasse invariant} $\HH$, a function on triples $(A,\iota,\omega)$, by setting $\HH(A,\iota,\omega) \in R$ determined by
  \begin{equation*}
    \fabs^*(\omega^\vee) = \HH(A,\iota,\omega) \omega^\vee \in H^1(A,\OO_A).
  \end{equation*}
  Note that $\fabs^*(\omega^\vee) \in e_2 \cdot H^1(A,\OO_A)$ because $\fabs$ commutes with $\OO_B$.
\end{defn}

\begin{lemma}
  The Hasse invariant $\HH$ is a modular form over $\fpbar$ with respect to $B^\times$, of weight $p-1$ and level 1.
\end{lemma}
\begin{proof}
  We compute for $(A,\iota)$ defined over an $\fpbar$-algebra $R$ and $\mu \in R^\times$,
  \begin{align*}\HH(A,\iota,\mu \omega) \cdot \mu^{-1} \omega^\vee & = \fabs^*(\mu^{-1} \omega^\vee)                  \\
                                                                 & = \mu^{-p} \fabs^*(\omega^\vee)                  \\
                                                                 & = \mu^{-p} \HH(A,\iota,\omega) \cdot \omega^\vee
  \end{align*}
  so $\HH(A,\iota,\mu \omega) = \mu^{1-p}\HH(A,\iota,\omega)$.
\end{proof}

%In the letter \cite{serre-letters} of Serre, a key property of the classical mod $p$ Hasse invariant used is that multiplication by it preserves mod $p$ Hecke eigenforms (nonzero mod $p$ modular forms $f$ that are eigenvectors for the Hecke operators at all primes not dividing $p$ or the level of $f$) and their eigenvalues. One can see this from the fact that the $q$-expansion of the classical Hasse invariant is 1.

We define mod $p$ Hecke eigenforms with respect to $B^\times$ as mod $p$ modular forms with respect to $B^\times$ that are eigenvectors for all the Hecke operators $T_\ell$ we defined above. In the setting of Shimura curves, modular forms do not have $q$-expansions, so we must provide a different proof that multiplication by the Hasse invariant with respect to $B^\times$ preserves Hecke eigenvalues.

\begin{prop}\label{Hasse quotient 2}
  Let $\HH$ be the Hasse invariant mod $p$ with respect to $B^\times$. Let $f$ be a weight $k$ mod $p$ modular form  with respect to $B^\times$ of level $V$. Then $f$ is a Hecke eigenform if and only if $\HH f$ is a Hecke eigenform. In this case, they have the same Hecke eigenvalues.
\end{prop}
\begin{proof}
  The Hecke operator $T_\ell$ is defined in terms of quadruples $(A,\iota,\alpha,\omega)$ by the formula
  \begin{equation*}
    T_\ell f (A,\iota,\alpha,\omega) = \frac{1}{\ell}\cdot \sum\limits_{\parbox{8em}{\scriptsize$\eta\st A \to A/G$ \\ $G \cong \ZZ/\ell\ZZ \times \ZZ/\ell\ZZ$ \\ as $\OO_B \otimes \ZZ/\ell\ZZ$-modules }} f(A/G, \iota, \alpha_\eta, (\eta^t)^*\omega)
  \end{equation*}
  and so acts on $\HH f$ by
  \begin{equation*}
    T_\ell (\HH f) (A,\iota,\alpha,\omega) = \frac{1}{\ell}\cdot \sum\limits_{\parbox{8em}{\scriptsize$\eta\st A \to A/G$ \\ $G \cong \ZZ/\ell\ZZ \times \ZZ/\ell\ZZ$ \\ as $\OO_B \otimes \ZZ/\ell\ZZ$-modules }} f(A/G, \iota, \alpha_\eta, (\eta^t)^*\omega) \cdot \HH(A/G, \iota, (\eta^t)^*\omega).
  \end{equation*}
  %where we dropped the level structure from $\HH$ because the Hasse invariant is a level 1 modular form, so does not depend on the level.
  We claim that
  \begin{equation}\label{eqn}
    \HH(A/G, \iota, (\eta^t)^*\omega)=\HH(A,\iota,\omega) \text{ for any isogeny $\eta\st A \to A/G$ of the form above. }
  \end{equation}
  Given \eqref{eqn}, if $f$ is an eigenvector for $T_\ell$, we see that $\HH f$ is an eigenvector for $T_\ell$ with the same eigenvalue. For the converse, if $\HH f$ is an eigenvector for $T_\ell$ with eigenvalue $a_\ell$, then for any quadruple $(A,\iota,\alpha,\omega)$ we have, given the claim \eqref{eqn},
  \begin{equation*}
    a_\ell \cdot \HH(A,\iota,\omega)f(A,\iota,\alpha,\omega) = T_\ell(\HH f)(A,\iota,\alpha,\omega) = \HH(A,\iota,\omega) \cdot T_\ell f(A,\iota,\alpha,\omega).
  \end{equation*}
  We deduce that $T_\ell f$ and $a_\ell f$ agree away from the zeros of the Hasse invariant. Since the Hasse invariant is nonzero there can only be finitely many of them, so $T_\ell f$ and $a_\ell f$ agree on a dense subset of the Shimura curve, so must coincide. It remains to prove the claim \eqref{eqn} of the Hasse invariant; we can express it as the requirement that the following diagram commutes
  % $$\xymatrix{e_2 \cdot H^1(A,\OO_{A}) \ar[r]^{(\eta^t)^*} \ar[d]^{F_{abs}^*} & e_2 \cdot H^1(A/G,\OO_{A/G}) \ar[d]^{F_{abs}^*} \\
  %     e_2 \cdot H^1(A,\OO_{A}) \ar[r]^{(\eta^t)^*} & e_2 \cdot H^1(A/G,\OO_{A/G}).}$$
  \begin{center}
    \begin{tikzcd}[column sep=large]
      e_2 \cdot H^1(A,\OO_{A}) \arrow[r, "(\eta^t)^*"] \arrow[d, "\fabs^*"] & e_2 \cdot H^1(A/G,\OO_{A/G}) \arrow[d, "\fabs^*"] \\
      e_2 \cdot H^1(A,\OO_{A}) \arrow[r, "(\eta^t)^*"] & e_2 \cdot H^1(A/G,\OO_{A/G}).
    \end{tikzcd}
  \end{center}

  This follows from the fact that the absolute Frobenius morphisms $\fabs$ commute with any isogeny, including $\eta^t$ (the absolute Frobenius is the identity on points, and the $p$-power map on structure sheaves commutes with morphisms of sheaves).
\end{proof}

The last proposition indicates that, insofar as the study of the Hecke eigenvalues on the spaces $M_k$ is concerned, it suffices to consider the quotients $W_k:=M_k/\HH M_{k-(p-1)}$.
We will do this after a detour through definite quaternion algebras.

\section{Algebraic modular forms for $D^\times$ (mod $p$)}\label{quat_alg}

In this section we switch our focus to the definite quaternion algebra $D$ over $\QQ$ with discriminant $p\delta$.
Following Serre and Gross, we consider algebraic modular forms on the algebraic group $D^\times$, which are locally constant functions on a finite double coset space of the adelic points of $D^\times$.
After defining these spaces, we delve into their structure as Hecke modules, guided by observations made by Serre~\cite{serre-letters} in the case $\delta=1$.

Before looking at the algebraic group $D^\times$ we give a detailed description of the situation in the case of the multiplicative group, as this will be required later.

\subsection{Review of locally constant functions for $\QQ^\times$}

The strong approximation theorem gives a decomposition of the ring of finite adeles over $\QQ$:
\begin{equation*}
  \af^\times = \QQ^\times_{>0} \cdot \prod_\ell \ZZ_\ell^\times\qquad\text{with }\QQ^\times_{>0} \cap\prod_\ell \ZZ_\ell^\times = 1.
\end{equation*}
\begin{lemma}
  For any $M\in\NN$ there is a canonical group homomorphism
  \begin{equation*}
    \theta\st \af^\times/\QQ^\times_{>0} \to \prod_\ell \left(\ZZ_\ell/\ell^{v_\ell(M)}\ZZ_\ell\right)^\times = \big(\ZZ/M\ZZ\big)^\times
  \end{equation*}
  with kernel
  \begin{equation*}
    \ker\theta =\prod_\ell \ZZ_\ell^\times(M)=:U(M),
    \quad\text{where}\quad
    \ZZ_\ell^\times(M)  =\left\{x\in\ZZ_\ell^\times\st x\equiv 1\mod{\ell^{v_\ell(M)}}\right\}.
  \end{equation*}
\end{lemma}
\begin{remark}
  Giving $(\ZZ/M\ZZ)^\times$ the discrete topology, any continuous homomorphism $\AA^\times/\QQ^\times \to (\ZZ/M\ZZ)^\times$ factors through $(\{\pm 1\} \times \af^\times)/ \QQ^\times \cong \af^\times/\QQ^\times_{>0}$, where $\RR^\times \to \{\pm 1\}$ is the sign map.
\end{remark}
\begin{proof}
  Let $\iota_\ell\st \QQ_\ell^\times\to\af^\times$ denote the canonical inclusion:
  \begin{equation*}
    \big(\iota_\ell(x)\big)_q =
    \begin{cases}
      1 & \text{if }q\neq\ell \\
      x & \text{if }q=\ell
    \end{cases}
    \qquad \text{for any }x\in\QQ_\ell^\times.
  \end{equation*}
  In the special case $x=\ell$, we write
  \begin{equation*}
    \iota_\ell(\ell)=\ell\cdot\alpha\in \QQ^\times\cdot\prod_q \ZZ_q^\times=\af^\times,
    \qquad\text{where }
    \alpha_q=\begin{cases}
      \frac{1}{\ell} & \text{if }q\neq\ell \\
      1              & \text{if }q=\ell.
    \end{cases}
  \end{equation*}
  Factoring $M=\ell_1^{e_1}\dots\ell_r^{e_r}$, the map $\theta$ is determined by the property that for $j=1,\dots, r$:
  \begin{align*}
    \theta(\iota_{\ell_j}(x))                       & = x\mod{\ell_j^{e_j}}                  &  & = \Big(1,\dots,1,x,1,\dots,1\Big)\qquad\qquad \text{if }x\in\ZZ_{\ell_j}^\times                \\
    \theta(\iota_{\ell_j}(\ell_j)) = \theta(\alpha) & = \frac{1}{\ell_j}\mod{M/\ell_j^{e_j}} &  & = \Big(\frac{1}{\ell_j},\dots,\frac{1}{\ell_j},1,\frac{1}{\ell_j},\dots,\frac{1}{\ell_j}\Big).
  \end{align*}
  To see that the diagonally embedded $\QQ^\times_{>0}\subseteq\af^\times$ is in the kernel of $\theta$, it suffices to consider an arbitrary prime $q$ (since they generate the group $\QQ^\times_{>0}$, and $\theta$ is a group homomorphism).

  \begin{center}
    \setlength{\tabcolsep}{12pt}
    \begin{tabular}{ll}
      \toprule
      $q=\ell_j$, WLOG $j=1$ & $q\neq \ell_1,\dots,\ell_r$                               \\ \midrule
      $\theta(\iota_{\ell_1}(q)) = \Big(1,\frac{1}{q},\frac{1}{q},\dots,\frac{1}{q}\Big)$
                             &
      $\theta(\iota_{\ell_1}(q))  = \Big(q,1,\dots,1\Big)$                               \\
      $\theta(\iota_{\ell_2}(q)) = \Big(1,q,1,\dots,1\Big)$
                             &
      \vdots                                                                             \\
      \vdots
                             &
      $\theta(\iota_{\ell_r}(q))  = \Big(1,\dots,1,q\Big)$                               \\

      $\theta(\iota_{\ell_r}(q))  = \Big(1,1,\dots,1,q\Big)$
                             &
      $\theta(\iota_{q}(q))       = \Big(\frac{1}{q},\dots,\frac{1}{q},\frac{1}{q}\Big)$ \\
      $\theta(\iota_{\ell}(q))    = \Big(1,1\dots,1,1\Big)\quad\text{at all other }\ell$
                             &
      $\theta(\iota_{\ell}(q))    = \Big(1,\dots,1,1\Big)\quad\text{at all other }\ell$. \\
      \bottomrule
    \end{tabular}
  \end{center}

  So in all cases $\theta(q)=1$.
\end{proof}

% Composing $\theta$ with a Dirichlet character $(\ZZ/M\ZZ)^\times \to \fpbar^\times$ gives a locally constant group homomorphism $\af^\times/\QQ^\times_{>0} \to \fpbar^\times$ that extends to a locally constant group homomorphism on $\AA^\times/\QQ^\times$.

We recall how a continuous character $\psi\st\gq\to\fpbar^\times$ gives rise to a Dirichlet character.
Since the image of $\psi$ is abelian, it factors through the abelianisation of $\gq$; since it is continuous, the fixed field of $\ker\psi$ is a finite abelian extension $K/\QQ$.
By the Kronecker--Weber Theorem, $K\subseteq\QQ(\zeta_M)$ for some $M\geq 1$, in which case we say that $\psi$ is defined modulo $M$, and we can think of $\psi$ as being a composition
\begin{equation*}
  \Gal\big(\QQ(\zeta_M)/\QQ\big)\to\Gal(K/\QQ)\to\fpbar^\times.
\end{equation*}
Since the Galois group of $\QQ(\zeta_M)/\QQ$ is isomorphic to $(\ZZ/M\ZZ)^\times$, $\psi$ can be identified with a (not necessarily primitive) Dirichlet character modulo $M$ that we also denote
\begin{equation*}
  \psi\st (\ZZ/M\ZZ)^\times\to\fpbar^\times.
\end{equation*}
The integers $M$ modulo which $\psi$ is defined form an ideal generated by the conductor $\cond(\psi)$.
\begin{lemma}\label{conductor_valuation}
  The conductor of $\psi$ satisfies $v_p(\cond(\psi))\in\{0,1\}$.
\end{lemma}
\begin{proof}
  Suppose $v_p(\cond(\psi))>0$ and write $\cond(\psi)=p^n M_0$ with $\gcd(p,M_0)=1$.
  The Chinese Remainder Theorem gives two group homomorphisms
  \begin{equation*}
    \varphi\st (\ZZ/p^n\ZZ)^\times\to\fpbar^\times\qquad\text{and}\qquad
    \varepsilon_0\st (\ZZ/M_0\ZZ)^\times\to\fpbar^\times.
  \end{equation*}
  Since $n>0$, $\varphi$ is non-trivial.
  As $p>2$, letting $g$ be a primitive root mod $p^n$, we have $\varphi(g)\in\FF_{p^m}^\times$ for some $m\in\NN$, so the order of $\varphi(g)$ must divide $(p-1)=\gcd\big((p-1)p^{n-1},p^m-1\big)$.
  Therefore $\varphi$ factors through $(\ZZ/p\ZZ)^\times$, so $\psi$ is defined modulo $pM_0$, and the minimality of the conductor gives $n=1$.
\end{proof}
For any integer $N$ coprime to $p$, $\psi$ is defined modulo $pN$ if and only if the prime-to-$p$ part of $\cond(\psi)$ divides $N$.
We can write $\psi=\chi^h\varepsilon$, where
\begin{equation*}
  \chi\st (\ZZ/p\ZZ)^\times\to\fpbar^\times,\quad \chi(x)=x,\quad h\in \ZZ/(p-1)\ZZ,\qquad\text{and}\qquad
  \varepsilon\st (\ZZ/N\ZZ)^\times\to\fpbar^\times.
\end{equation*}
Here $\chi$ corresponds to the mod $p$ cyclotomic character, and $\varepsilon$ is a (not necessarily primitive) Dirichlet character modulo $N$.

\begin{prop}\label{galois_char_loc_const}
  Any continuous character $\psi\st\gq\to\fpbar^\times$ gives rise to a locally constant group homomorphism $\psi'\st \af^\times/\QQ^\times_{>0}\to\fpbar^\times$ such that $U(M)\subseteq\ker\psi'$ for every $M$ modulo which $\psi$ is defined.
\end{prop}
\begin{proof}
  Identifying $(\ZZ/M\ZZ)^\times$ with $\Gal(\QQ(\zeta_M)/\QQ)$, the map $\AA^\times/\QQ^\times \to (\ZZ/M\ZZ)^\times$ factoring through $\theta$ becomes the Artin reciprocity map.
  Composing $\theta$ with $\psi^{-1}$ we get
  \begin{equation*}
    \psi'\st \af^\times/\QQ^\times_{>0} \xrightarrow{\theta} \big(\ZZ/M\ZZ\big)^\times \xrightarrow{\psi^{-1}} \fpbar^\times.
  \end{equation*}
  The fact that $\ker\theta=U(M)$ allows us to conclude that $U(M)\subseteq\ker \psi'$.

  If $M\mid M'$, then $\theta$ for $M$ is the composition of $\theta$ for $M'$ with the map $(\ZZ/M'\ZZ)^\times\to (\ZZ/M\ZZ)^\times$, so $\psi'$ is independent of the choice of $M$.
\end{proof}

\subsection{The quaternion algebra $D$}

We turn now to the unique \emph{definite} quaternion algebra $D$ over $\QQ$ with discriminant $p\delta$.
The units of $D$ define an algebraic group $D^\times$ over $\QQ$:
\begin{equation*}
  D^\times(R)=\big(D\otimes_\QQ R\big)^\times\qquad\text{for any $\QQ$-algebra $R$}.
\end{equation*}
Fix a maximal order $\oh$ of $D$.
(For the purposes of comparison with the Shimura curve setup, we will later need to choose a particular maximal order $\oh$ as the endomorphism ring of a distinguished object, see the proof of \autoref{main thm 2}.)

For a prime $\ell$ we write $\oh_\ell:=\oh \otimes \ZZ_\ell$ and $D_\ell :=D \otimes \QQ_\ell$.
For $\ell\nmid p\delta$, we identify $\oh_\ell\cong\Mat_2(\ZZ_\ell)$ via a fixed isomorphism; the reduced norm $\Nm_\ell\st D_\ell\to \QQ_\ell$ is simply $\Nm_\ell(x)=\det(x)$ for $x\in D_\ell\cong \Mat_2(\QQ_\ell)$.

The division algebra $D_p$ can be presented as $D_p=\QQ_p\oplus \QQ_p i\oplus \QQ_p \pi\oplus \QQ_p i\pi$, where
\begin{equation*}
  i^2=e=\text{a quadratic non-residue mod $p$},\qquad \pi^2=p,\qquad \pi i=-i\pi.
\end{equation*}
The maximal order is then
\begin{equation*}
  \oh_p=\ZZ_p\oplus \ZZ_p i\oplus \ZZ_p \pi\oplus \ZZ_p i\pi,
\end{equation*}
and $\pi$ is a uniformiser for the unique maximal ideal of $\oh_p$.

The reduced norm $\Nm_p\st D_p\to \QQ_p$ is given by
\begin{equation*}
  \Nm_p(a+bi+c\pi+di\pi)=a^2-eb^2-pc^2+epd^2.
\end{equation*}
For $x\in D_p$, we have $x\in\oh_p^\times$ if and only if $\Nm_p(x)\in\ZZ_p^\times$.

(Proofs of the above facts, as well as more details about the division algebra $D_p$ can be found in \cite[Section II.1]{Vigneras} or \cite[Theorem 13.1.6]{Voi1}.)

For any integer $N$ coprime to $p\delta$ and any $\ell\nmid p\delta$ define
\begin{align*}
  \oh_\ell^\times(N)     & =\oh_{\ell,1}^\times(N)=\oh_\ell^\times                                                                                                                                                            & \text{if }\ell\nmid N \\
  \oh_\ell^\times(N)     & = \left\{\begin{pmatrix}a & b\\c & d\end{pmatrix}\in\GL_2(\ZZ_\ell)\st \begin{pmatrix}a & b\\c & d\end{pmatrix}\equiv\begin{pmatrix}1 & 0\\0 & 1\end{pmatrix}\mod{\ell^{v_\ell(N)}}\right\}        & \text{if }\ell\mid N  \\
  \oh_{\ell,1}^\times(N) & = \left\{\begin{pmatrix}a & b\\c & d\end{pmatrix}\in\GL_2(\ZZ_\ell)\st \begin{pmatrix}a & b\\c & d\end{pmatrix}\equiv\begin{pmatrix}1 & \ast \\0 & \ast\end{pmatrix}\mod{\ell^{v_\ell(N)}}\right\} & \text{if }\ell\mid N  \\
  \oh_p^\times(\pi)      & = \left\{x\in\oh_p^\times\st x\equiv 1\mod{\pi}\right\}.
\end{align*}

\begin{lemma}\label{norm_at_p}
  The reduced norm $\Nm_p\st D_p\to \QQ_p$ maps $\oh_p^\times(\pi)$ to $\ZZ_p^\times(p)$.
\end{lemma}
\begin{proof}
  If $x\in\oh_p^\times(\pi)$ then $x=1+y\pi$ for some $y\in\oh_p$; writing $y=a+bi+c\pi+di\pi$ with $a,b,c,d\in\ZZ_p$, we have
  \begin{equation*}
    \Nm_p(x)=\Nm_p\big((1+pc)-pdi+a\pi-bi\pi\big)=(1+pc)^2-ep^2d^2-pa^2+peb^2\equiv 1\pmod{p\ZZ_p},
  \end{equation*}
  hence $\Nm_p(x)\in\ZZ_p^\times(p)$.
\end{proof}

\begin{lemma}\label{dihedral}
  For any $\mu\in\oh_p^\times$, we have
  \begin{equation*}
    \pi^{-1}\mu\pi\in\oh_p^\times\qquad\text{and}\qquad \pi^{-1}\mu\pi\mu^{-p}\in\oh_p^\times(\pi).
  \end{equation*}
\end{lemma}
\begin{proof}
  The first statement is clear since $\Nm_p(\pi^{-1}\mu\pi)=\Nm_p(\mu)$.

  For the second, write $\varphi\st\oh_p^\times\to\fps^\times$ for the reduction modulo $\pi$.
  We want to show that
  \begin{equation*}
    \varphi(\pi^{-1}\mu\pi)=\varphi(\mu)^p\qquad\text{in }\fps.
  \end{equation*}
  In the above basis, we have
  \begin{equation*}
    \mu=a+bi+c\pi+di\pi,\qquad a,b,c,d\in\ZZ_p,
  \end{equation*}
  therefore
  \begin{equation*}
    \pi^{-1}\mu\pi=a-bi+c\pi-di\pi.
  \end{equation*}
  After reduction modulo $\pi$, we have
  \begin{equation*}
    \varphi(\mu)=a+bi\qquad\text{and}\qquad\varphi(\pi^{-1}\mu\pi)=a-bi,
  \end{equation*}
  where $a,b\in\FF_p$ and $\pm i$ are the two roots of the irreducible quadratic polynomial $x^2-e$.
  Hence $\varphi(\mu)$ and $\varphi(\pi^{-1}\mu\pi)$ are nontrivial Galois conjugates in $\fps$, but the Galois conjugation $i\mapsto -i$ is raising to the power $p$ in $\fps$.
\end{proof}

\subsection{Algebraic modular forms for $D^\times$}

For any prime $\ell$, we denote $\iota_\ell\st D_\ell^\times\to D^\times(\af)$ the canonical inclusion.
%(There should be no confusion with the analogous map $\iota_\ell\st\QQ_\ell^\times\to\af^\times$.)
Set
\begin{align*}
  U^D(N)        & =\oh_p^\times(\pi)\times\prod_{\ell\neq p} \oh_\ell^\times(N)                          \\
  U^D_1(N)      & =\oh_p^\times(\pi)\times\prod_{\ell\neq p} \oh_{\ell,1}^\times(N)                      \\
  \Omega^D(N)   & =U^D(N)\backslash D^\times(\af)/D^\times(\QQ)                                          \\
  \Omega^D_1(N) & =U^D_1(N)\backslash D^\times(\af)/D^\times(\QQ)                                        \\
  \MM(\Omega)   & =\left\{f\st \Omega\to\fpbar\right\}\qquad\text{for }\Omega=\Omega^D(N),\Omega^D_1(N).
  %M(\Omega)    & =\left\{f\st D^\times(\af)/D^\times(\QQ)\to\fpbar\st f\text{ locally constant}\right\}.
\end{align*}
%(We might want to continue to require $\oh_p^\times(1)$-invariance in $M(\Omega)$, or simply work with $M(\Omega(N))$.)
The latter $\fpbar$-vector space $\MM(\Omega)$ is called the space of algebraic modular forms of level $\Omega$ on $D^\times$.
It is equipped with Hecke operators $T_\ell$ (recall, for $\ell \nmid Np\delta$):
\begin{equation*}
  T_\ell f (x) = \frac{1}{\ell} \sum_g f\big(\iota_\ell(g)x\big),\qquad\text{where}\qquad
  \GL_2(\ZZ_\ell)\begin{pmatrix}1 & 0 \\ 0 & \ell\end{pmatrix}\GL_2(\ZZ_\ell)=\bigsqcup_g \GL_2(\ZZ_\ell)g.
\end{equation*}

There is a left action of the group $D^\times(\af)$ on the space $\MM(\Omega)$ by
\begin{equation*}
  (g\cdot f)(x)=f\big(g^{-1}x\big).
\end{equation*}
% Can also do $(g\cdot f)(x)=f(xg)$.
As Serre notes in~\cite{serre-letters}, this action has some very interesting properties.

To see this, consider the subspace of $\MM(\Omega)$ of functions of weight $k\in\ZZ$, given by
\begin{equation*}
  \MM_k(\Omega)=\left\{f\st\Omega\to\fpbar\st f\big(\iota_p(\mu)x\big)=[\mu]^{-k}f(x)\text{ for all }[\mu]\in\fps^\times\right\},
\end{equation*}
where $\mu\in\oh_p^\times$ is any representative of $[\mu]\in\fps^\times=\oh_p^\times/\oh_p^\times(\pi)$.
Since $[\mu]^{p^2-1}=1\in\fps^\times$, we have
\begin{prop}\label{MM_periodic}
  $\MM_{k+p^2-1}(\Omega)=\MM_k(\Omega)$ for all $k\in\ZZ$.
\end{prop}
Note that the Hecke operators $T_\ell$ send functions of weight $k$ to functions of weight $k$ as multiplication by $\iota_\ell(g)$ does not change the component at $p$.

The following is a generalisation of~\cite[Comment (11)]{serre-letters}:
\begin{prop}\label{MM_dihedral}
  The map $f\mapsto \iota_p(\pi)\cdot f$ defines a Hecke-equivariant isomorphism
  \begin{equation*}
    \MM_k(\Omega)\to \MM_{pk}(\Omega).
  \end{equation*}
\end{prop}
\begin{proof}
  Let $f\in \MM_k(\Omega)$ and let $g=\iota_p(\pi)\cdot f$.
  For any $\mu\in\oh_p^\times$, let $\kappa=\pi^{-1}\mu\pi\mu^{-p}$, then \autoref{dihedral} tells us that $\kappa\in\oh_p^\times(\pi)$, so
  \begin{align*}
    g(\iota_p(\mu)x) & =f\big(\iota_p(\pi^{-1}\mu)x\big)                           \\
                     & =f\big(\iota_p(\kappa\mu^p\pi^{-1})x\big)                   \\
                     & =f\big(\iota_p(\kappa)\iota_p(\mu^p)\iota_p(\pi)^{-1}x\big) \\
                     & =[\mu]^{-pk}f\big(\iota_p(\pi)^{-1}x\big)                   \\
                     & =[\mu]^{-pk}g(x),
  \end{align*}
  where we used the $\oh_p^\times(\pi)$-invariance of $f$, and the fact that $f$ has weight $k$.

  The linearity of the map is clear.
  For the bijectivity, note first that the map is injective, since it is the restriction of a bijective map $\MM(\Omega)\to \MM(\Omega)$ to the subspace $\MM_k(\Omega)$.
  Therefore the composition
  \begin{equation*}
    \MM_k(\Omega)\to \MM_{pk}(\Omega)\to \MM_{p^2k}(\Omega)=\MM_k(\Omega)
  \end{equation*}
  is an injective endomorphism of a finite-dimensional vector space, hence bijective.

  The Hecke equivariance is simply due to the fact that for any prime $\ell\nmid p\delta$, the action of the Hecke operator $T_\ell$ is defined by working entirely in the $\QQ_\ell$ component of $\af$ via $\iota_\ell$, and the action of $\pi$ is defined by working entirely in the $\QQ_p$ component via $\iota_p$, hence these two actions commute.
\end{proof}

% The centre $Z=\af^\times$ of $D^\times(\af)$ acts on $\MM(\Omega)$ by restricting the action of $D^\times(\af)$.
% Given $\omega\st \af^\times\to\fpbar^\times$, the subspace of functions with central character $\omega$ is defined to be
% \begin{equation*}
%   \MM(\Omega;\omega)=\left\{f\st\Omega\to\fpbar\st f\big(z^{-1}x\big)=\omega(z)f(x)\text{ for all }z\in Z\right\}.
% \end{equation*}

\subsection{Twisting Hecke eigensystems}

It remains to discuss one more structural relation between spaces of algebraic modular forms, given by twisting by a Galois character.
This generalises~\cite[Comment (13)]{serre-letters}.

We first describe how a Galois character gives an algebraic modular form on $D^\times$.
We package the local reduced norms $\Nm_\ell\st D^\times_\ell\to\QQ^\times_\ell$ into an adelic map $\Nm\st D^\times(\af)/D^\times(\QQ)\to\af^\times/\QQ^\times_{>0}$.

Let $\psi\st\gq\to\fpbar^\times$ be a Galois character defined modulo $pN$; by \autoref{conductor_valuation}, this means that the prime-to-$p$ part of $\cond(\psi)$ divides $N$.
Write $\psi=\chi^h\varepsilon$, with $\varepsilon$ a character modulo $N$, and let $\psi'\st\af^\times/\QQ^\times_{>0}\to\fpbar^\times$ be the corresponding locally constant homomorphism from \autoref{galois_char_loc_const}, so that $U(pN)\subseteq\ker\psi'$.

Consider the composition
\begin{equation*}
  \psi_D\st D^\times(\af)/D^\times(\QQ)\xrightarrow{\Nm} \af^\times/\QQ^\times_{>0} \xrightarrow{\psi'} \fpbar^\times.
\end{equation*}
Note that $\psi_D$ is multiplicative:
\begin{equation*}
  \psi_D(xy)=\psi_D(x)\psi_D(y)\qquad\text{for all }x,y\in D^\times(\af).
\end{equation*}

\begin{prop}\label{psiD}
  We have $\Nm\big(U^D(N)\big)\subseteq U(pN)$ and $\psi_D\in \MM_{h(p+1)}\big(\Omega^D(N)\big)$.

  If $\varepsilon=\mathbf{1}$ then $\psi_D\in \MM_{h(p+1)}\big(\Omega^D_1(N)\big)$.
\end{prop}
\begin{proof}
  If $\ell\nmid pN$ then $\ZZ_{\ell}^\times(pN)=\ZZ_{\ell}^\times$ and there is nothing to check.

  If $\ell\mid N$ then $\oh_\ell^\times\cong\GL_2(\ZZ_\ell)$ and the local reduced norm map $\Nm_\ell$ is the determinant, so
  \begin{align*}
    x_\ell=\begin{pmatrix} a_\ell & b_\ell \\ c_\ell & d_\ell\end{pmatrix}\in\oh_\ell^\times(N)
     & \Rightarrow \begin{pmatrix} a_\ell & b_\ell \\ c_\ell & d_\ell\end{pmatrix}\equiv\begin{pmatrix} 1 & 0 \\ 0 & 1\end{pmatrix}\pmod{\ell^{v_\ell(N)}} \\
     & \Rightarrow a_\ell d_\ell - b_\ell c_\ell \equiv 1\pmod{\ell^{v_\ell(N)}}                                                                           \\
     & \Rightarrow \Nm_\ell(x_\ell)=\det(x_\ell)\in \ZZ_\ell^\times(pN).
  \end{align*}
  Finally, if $\ell=p$ we use \autoref{norm_at_p} to get that $\Nm_p$ maps $\oh_p^\times(\pi)$ to $\ZZ_p^\times(p)\subseteq\ZZ_p^\times(pN)$.

  We conclude that $\Nm\big(U^D(N)\big)\subseteq U(pN)\subseteq\ker\psi'$.

  Since $\ZZ_{\ell}^\times(p)=\ZZ_{\ell}^\times$ for all $\ell\neq p$, we get that $\Nm\big(U_1^D(N)\big)\subseteq U(p)$.
  If $\varepsilon=\mathbf{1}$ then $\psi$ is defined modulo $p$, so $U(p)\subseteq\ker\psi'$.

  Since $\psi_D$ is multiplicative, to find its weight it suffices to compute for $\mu\in\oh_p^\times$
  \begin{equation*}
    \psi_D\big(\iota_p(\mu)\big)=\psi^{-1}\left(\theta\big(\iota_p(\Nm_p(\mu))\big)\right).
  \end{equation*}
  As before, we write $[\mu]$ for the image of $\mu\in\oh_p^\times$ in $\fps^\times=\oh_p^\times/\oh_p^\times(\pi)$.
  If $\mu=a+bi+c\pi+di\pi$, then $[\mu]=[a+bi]$.

  Let us write $[\Nm_p(\mu)]$ for the image of $\Nm_p(\mu)\in\ZZ_p^\times$ in $\ZZ_p^\times/\ZZ_p^\times(p)=\fp^\times\subseteq\fps^\times$.
  Then
  \begin{align*}
    \Nm_p(\mu) & =a^2-eb^2-pc^2+epd^2                                                                           \\
               & \Rightarrow\quad [\Nm_p(\mu)]=[a^2-eb^2]=[a-bi][a+bi]=[a+bi]^p[a+bi]=[a+bi]^{p+1}=[\mu]^{p+1},
  \end{align*}
  since the nontrivial Galois conjugation $[a+bi]\mapsto [a-bi]$ of $\fps^\times$ is given by raising to the power $p$ Frobenius.

  Going back to $\psi_D$, we observe that
  \begin{equation*}
    \psi_D\big(\iota_p(\mu)\big)=\psi^{-1}\left(\theta\big(\iota_p(\Nm_p(\mu))\big)\right)
    =\chi^{-h}\big([\Nm_p(\mu)]\big)=[\mu]^{-h(p+1)},
  \end{equation*}
  hence $\psi_D$ has weight $h(p+1)$.
\end{proof}

% This is now subsumed by the previous result.
% \begin{lemma}
%   If the conductor $M=N$ of $\psi'$ is coprime to $p$, then $\psi_D$ has weight $0$:
%   \begin{equation*}
%     \psi_D\in M_0\big(\Omega^D(N)\big).
%   \end{equation*}
% \end{lemma}
% \begin{proof}
%   If $p\nmid N$ then $\Nm_p(\oh_p^\times)\subseteq \ZZ_p^\times\subseteq\ker\psi'$, so
%   \begin{equation*}
%     \psi_D(\iota_p(\mu)x)=\psi_D(\iota_p(\mu))\psi_D(x)=\psi_D(x)\qquad\text{for all }\mu\in\oh_p^\times.\qedhere
%   \end{equation*}
% \end{proof}

% The central character of $\psi_D$ is $\psi^2$.

\begin{lemma}\label{hecke_prod}
  If $f$ is a locally constant function on $D^\times(\af)/D^\times(\QQ)$, then
  \begin{equation*}
    T_\ell\big(f \psi_D\big) = \psi(\ell)  \psi_D T_\ell(f)
  \end{equation*}
  for all primes $\ell$ coprime to $p\delta N$ and the level of $f$.

  In particular, $\psi_D$ is an eigenfunction of $T_\ell$ for all primes coprime to $p\delta N$:
  \begin{equation*}
    T_\ell \psi_D = \big(1+\ell^{-1}\big)\psi(\ell) \psi_D.
  \end{equation*}
\end{lemma}
\begin{proof}
  Recall that the definition of $T_\ell$ is a sum over representatives $g\in\GL_2(\QQ_\ell)$ given by
  \begin{equation*}
    \GL_2(\ZZ_\ell)\begin{pmatrix}1 & 0 \\ 0 & \ell\end{pmatrix}\GL_2(\ZZ_\ell)=\bigsqcup_g \GL_2(\ZZ_\ell)g,
  \end{equation*}
  in particular all such $g$ have $\Nm_\ell(g)=\det(g)=\ell$.
  Therefore
  \begin{equation*}
    \psi_D\big(\iota_\ell(g)\big)=\psi^{-1}\left(\theta\big(\iota_\ell(\Nm_\ell(g))\big)\right)=\psi^{-1}\left(\theta(\iota_\ell(\ell))\right)=\psi^{-1}(\ell^{-1})=\psi(\ell).
  \end{equation*}
  We have
  \begin{align*}
    (T_\ell (\psi_D f))(x)
     & =\frac{1}{\ell}\sum_g f\big(\iota_\ell(g)x\big)\psi_D\big(\iota_\ell(g)x\big)         \\
     & =\frac{1}{\ell}\sum_g f\big(\iota_\ell(g)x\big)\psi_D\big(\iota_\ell(g)\big)\psi_D(x) \\
     & =\frac{1}{\ell}\psi(\ell)\psi_D(x)\sum_g f\big(\iota_\ell(g)x\big)                    \\
     & =\psi(\ell)\psi_D(x) (T_\ell f)(x).\qedhere
  \end{align*}
\end{proof}

% There is a reducible Galois representation corresponding to the Hecke eigensystem of $\psi_D$, namely $\chi^{-1}\psi\oplus\psi$.

Given a Hecke module $\MM$ and a character $\psi$, we denote by $\MM[\psi]$ the twist of $\MM$ by $\psi$: this is the same as $\MM$ as a vector space, but the action of $T_\ell$ is replaced by $T_\ell \psi(\ell)$.

\begin{thm}\label{twist_iso}
  Let $\psi=\chi^h\varepsilon$ be a character defined modulo $pN$.
  Multiplication by $\psi_D$ is a Hecke-equivariant isomorphism
  \begin{equation*}
    \MM_k\big(\Omega^D(N)\big)[\psi]\to \MM_{k+h(p+1)}\big(\Omega^D(N)\big).
  \end{equation*}

  If $\varepsilon=\mathbf{1}$, then multiplication by $\psi_D$ is a Hecke-equivariant isomorphism
  \begin{equation*}
    \MM_k\big(\Omega_1^D(N)\big)[\psi]\to \MM_{k+h(p+1)}\big(\Omega_1^D(N)\big).
  \end{equation*}
\end{thm}
\begin{proof}
  Let $\Omega=\Omega^D(N)$, or $\Omega=\Omega_1^D(N)$ if $\varepsilon=\mathbf{1}$.

  % We start by noting that $\psi_D$ takes only nonzero values, so multiplication by $\psi_D$ is injective.
  % The Hecke equivariance is given by \autoref{hecke_prod}, so it remains to prove surjectivity.

  % Let $a\geq 1$ be such that $ah$ is divisible by $p-1$, so that $k+ah(p+1)\equiv k\mod{p^2-1}$.
  % Taking the corresponding $a$-fold composition of maps we get
  % \begin{equation*}
  %   \MM_k(\Omega)[\psi^a]\to \MM_{k+ah(p+1)}(\Omega)=\MM_k(\Omega),
  % \end{equation*}
  % which is a bijection, so our original map is a bijection as well.

  From \autoref{psiD} and \autoref{hecke_prod} we know that multiplication by $\psi_D$ is a Hecke-equivariant homomorphism $\MM_k(\Omega)[\psi]\to \MM_{k+h(p+1)}(\Omega)$.

  Noting that we may apply the same results to $\psi^{-1}=\chi^{-h}\varepsilon^{-1}$, we get that multiplication by $(\psi^{-1})_D$ is the inverse to multiplication by $\psi_D$.
\end{proof}

A crucial special case is that of the mod $p$ cyclotomic character $\chi$; it is defined modulo $p$, hence modulo $pN$ for every $N$, with $\varepsilon=\mathbf{1}$.
Therefore
\begin{cor}\label{twist_by_cyclo}
  Let $\Omega=\Omega^D(N),\Omega_1^D(N)$.
  Multiplication by $\chi_D\in\MM_{p+1}(\Omega)$ is a Hecke-equivariant isomorphism
  \begin{equation*}
    \MM_k(\Omega)[\chi]\to \MM_{k+p+1}(\Omega).
  \end{equation*}
\end{cor}

% \begin{remark}\label{twist_by_cyclo_level_1}
%   As $\chi_D$ has level $1$, we also have $\chi_D\in \MM_{p+1}(\Omega^D_1(1))$ and obtain, as above, a Hecke-equivariant isomorphism $\MM_k(\Omega^D_1(N))[\chi]\to \MM_{k+p+1}(\Omega^D_1(N))$.
% \end{remark}

Combining \autoref{MM_periodic}, \autoref{MM_dihedral}, and \autoref{twist_by_cyclo}, we obtain
\begin{thmb}
  Let $\Omega=\Omega^D(N),\Omega_1^D(N)$.
  For every $k\in\ZZ$ there are explicit Hecke-equivariant isomorphisms
  \begin{equation*}
    \MM_k(\Omega)\cong \MM_{k+p^2-1}(\Omega),\quad
    \MM_k(\Omega)\cong \MM_{pk}(\Omega),\quad
    \MM_k(\Omega)[\chi]\cong \MM_{k+p+1}(\Omega),
  \end{equation*}
  where $\chi$ denotes the mod $p$ cyclotomic character.
\end{thmb}

\section{A generalised Deuring correspondence}

%In the previous section, we defined systems of Hecke eigenvalues associated to automorphic forms for the quaternion algebra $B^\times$ in a geometric way - by associating to $B^\times$ a Shimura curve $\XX$. On the other hand, our study of locally constant functions for $D^\times$ has been purely algebraic. Here we give a geometric interpretation of these locally constant functions for $D^\times$ compatible with that for $B^\times$. More precisely, our mod $p$ locally constant functions for $D^\times$ can be viewed as mod $p$ automorphic forms on a certain Shimura subvariety of $\XX$, its supersingular locus. We begin by describing properties of supersingular false elliptic curves, then prove a Deuring correspondence relating the sets $\Omega^D(N), \Omega^D_1(N)$ to the supersingular locus of $\XX(N),\XX_1(N)$.

Following the ideas of \cite{serre-letters}, we consider the restrictions of the mod $p$ modular forms with respect to $B^\times$ to the supersingular locus of the Shimura curve $\XX$. This supersingular locus is a Shimura subvariety closed under the Hecke operators on $\XX$. The restrictions of the mod $p$ modular forms with respect to $B^\times$ may then be viewed as mod $p$ automorphic forms for $D^\times$; these are locally constant functions defined on $D^\times(\AA)$. As $D^\times(\RR)$ is connected, these locally constant functions factor through $D^\times(\AA_f)$ and may be viewed as functions in $\MM(\Omega)$ for an appropriate $\Omega$. This is explicitly realised through a generalisation of the classical Deuring correspondence for supersingular elliptic curves.

The supersingular false elliptic curves are defined as the zero locus of the Hasse invariant for $B^\times$. For a more elementary definition, we make use of \cite[Proposition 5.2]{Mil1}:

\begin{prop}
  Any false elliptic curve over $\fpbar$ is isogenous to the square of an elliptic curve over $\fpbar$.
\end{prop}

\begin{defn}
  We say that a false elliptic curve $(A,\iota)$ over $\fpbar$ is supersingular if $A$ is isogenous to the square of a supersingular elliptic curve.
\end{defn}

\begin{prop}\label{super false prop}
  We have the following properties:
  \begin{enumerate}[label=(\roman*)]
    \item The isogeny class of any supersingular false elliptic curve consists exactly of the supersingular false elliptic curves.
    \item The endomorphism algebra $\End_{\OO_B}(A) \otimes \QQ$ of a supersingular false elliptic curve $(A,\iota)$ is isomorphic to the quaternion algebra $D$.
    \item Every supersingular false elliptic curve $(A,\iota)$ has a canonical $\fps$-structure where the $p^2$ Frobenius endomorphism acts by multiplication by $[-p]$. In other words, $(A,\iota)$ is isomorphic (over $\fpbar$) to (the base change to $\fpbar$ of) a supersingular false elliptic curve $(A',\iota')$ defined over $\fps$ on which the $p^2$ Frobenius endomorphism acts by $[-p]$.
    \item The Hasse invariant $\HH$ with respect to $B^\times$ has (simple) zeros exactly at the supersingular false elliptic curves.
  \end{enumerate}
\end{prop}
\begin{proof}
  \begin{enumerate}[label=(\roman*)]
    \item This follows from the analogous fact for supersingular elliptic curves, as well as the fact that being isogenous is an equivalence relation on false elliptic curves.
      %(it is an equivalence relation on abelian varieties, and by an application of the Skolem-Noether theorem, if the underlying abelian varieties of the false elliptic curves are isogenous then there exists an isogeny of false elliptic curves - see \cite{Mil1} page 179).
    \item In more generality see \cite[Proposition 5.2]{Mil1}. Suppose $(A,\iota)$ is isogenous to $E^2$ for some supersingular elliptic curve $E$. Then $E$ has endomorphism algebra $D'$, where $D'$ is the unique quaternion algebra over $\QQ$ ramified at $\{p,\infty\}$, and so we have $\End(A) \otimes \QQ \cong \Mat_2(D')$. Applying the double centraliser theorem to $B \subset \End(A) \otimes \QQ$ and its centraliser $\End_{\OO_B}(A) \otimes \QQ$, we deduce that
      \begin{equation*}
        B \otimes_{\QQ} (\End_{\OO_B}(A) \otimes \QQ) \cong \Mat_2(D').
      \end{equation*}
      In other words, in the Brauer group $\Br(\QQ)$ we have $[B] \cdot [\End_{\OO_B}(A) \otimes \QQ] = [D']$. Thus $\End_{\OO_B}(A) \otimes \QQ$ is the unique definite quaternion algebra over $\QQ$ ramified at the primes dividing $p\delta $.
    \item Our argument is based on the proof of \cite[Lemma 3.2.1]{BGJGP1}. Let $E$ be any supersingular elliptic curve with canonical $\fps$-structure (in fact by Honda-Tate theory, one can find such an $E$ defined over $\fp$). Then $E_{\fpbar}^2$ is a supersingular false elliptic curve under any choice of embedding $\OO_B \hookrightarrow \Mat_2(\OO_D)$. There exists an isogeny $\phi\st E_{\fpbar}^2 \to A$ of supersingular false elliptic curves by part (i). By the proof of \cite[Lemma 9]{DT1} we can assume that $\phi$ has degree coprime to $p$. Then the kernel $K$ of $\phi$ is stable both under the action of $\OO_B \subset \End(E_{\fpbar}^2)$ (by definition of $\phi$ as an isogeny of false elliptic curves) and also under $[-p]$, the square of the Frobenius. Hence $K$ is defined over $\fps$ and $(A' = E^2/K,\iota)$ is a false elliptic curve over $\fps$ with $A_{\fpbar}' \cong A$ an isomorphism of supersingular false elliptic curves (we require $\phi$ to be separable for the induced map $A_{\fpbar}' \to A$ to be an isomorphism). The $p^2$ Frobenius acts on $A'$ by $[-p]$.

    \item By definition, the Hasse invariant with respect to $B^\times$ vanishes at a false elliptic curve $(A,\iota)$ if and only if $F_{abs}^*$ is the zero map $e_2 \cdot H^1(A,\OO_A) \to e_2 \cdot H^1(A,\OO_A)$. Since the idempotents $e_1,e_2$ are conjugate, this occurs if and only if $F_{abs}^*$ is zero on $H^1(A,\OO_A)$. Now $A$ is isogenous to $E^2$ for some elliptic curve $E$, and since the absolute Frobenius commutes with isogenies, we see that $F_{abs}^*$ is zero on $H^1(A,\OO_A)$ if and only if it is zero on $H^1(E^2,\OO_{E^2})$. Then this is zero if and only if $E$ is supersingular by the Kunneth formula. The zeros are simple by \cite[Lemma 5.2]{Kas1}.\qedhere
      % stated for V_1 structure but reference in Diamond Taylor works for any U-structure.
  \end{enumerate}
\end{proof}

\begin{remark}
  Since multiplication by $[-p]$ commutes with any isogeny, (iii) tells us that any isogeny of supersingular false elliptic curves with canonical $\fps$-structure is defined over $\fps$.
\end{remark}

Something stronger than part (iii) above holds:

\begin{prop}
  Any supersingular false elliptic curve $(A,\iota)$ over $\fpbar$ is superspecial, meaning it is \textbf{isomorphic} to the square of a supersingular elliptic curve.
\end{prop}
\begin{proof}
  This is based on arguments in \cite{Ghitza2} and \cite{Phi1}. It is equivalent to show that the $a$-number $a(A)=\dim_{\fpbar}\Hom(\alpha_p,A)$ is $2$, where $\alpha_p$ is the finite group scheme $\Spec \fpbar[x]/(x^p)$. Since $A$ is supersingular, we have $a(A) \geq 1$, and since $A$ is of dimension $2$, $a(A) \leq 2$. But $\Hom(\alpha_p,A)$ is a module under $\End(A) \otimes \fpbar$, and hence under $\OO_B \otimes \fpbar \cong \Mat_2(\fpbar)$. But for any field $k$, $k$ is not a $\Mat_2(k)$-module.
\end{proof}
% Hence there is only 1 A up to isomorphism, but still have different choices of iota

We recall the Isogeny Theorem of Tate for abelian varieties.
\begin{defn}
  Let $A$ be an abelian variety over a field $K$ of characteristic prime to $\ell$. We define the $\ell$-adic Tate module of $A$ to be $\Ta_\ell(A) := \varprojlim\limits_n A[\ell^n](\overline{K})$.
\end{defn}

\begin{thm}
  Let $A,A_0$ be abelian varieties over a finite field $K$ of characteristic prime to $\ell$. The natural map
  \begin{equation*}
    \Hom_K(A,A_0) \otimes \ZZ_\ell \to \Hom_{\Gal(\overline{K}/K)}(\Ta_\ell(A),\Ta_\ell(A_0))
  \end{equation*}
  is an isomorphism.
\end{thm}
\begin{proof}
  \cite{Tat1}.
\end{proof}

\begin{cor}\label{isog thm false}
  Let $(A,\iota), (A_0,\iota_0)$ be supersingular false elliptic curves with canonical structure over $\fps$. On the underlying abelian varieties we have, for $\ell \neq p$,
  \begin{equation*}
    \Hom_{\fps}(A,A_0) \otimes \ZZ_\ell \cong \Hom_{\Gal(\fpbar/\fps)}(\Ta_\ell(A), \Ta_\ell(A_0)) \cong \Mat_4(\ZZ_\ell).
  \end{equation*}
\end{cor}

%%%%%%%%%%%%%%%%%%%%%%%%%%%%

\begin{cor}
  Let $(A_0,\iota_0)$ be a supersingular false elliptic curve with canonical structure over $\fps$, and let $\OO= \End_{\OO_B}(A_0)$. For any $\ell \nmid p\delta$, fix an isomorphism $\Ta_\ell(A_0) \cong \Mat_2(\ZZ_\ell)$ such that $\OO_B \otimes \ZZ_\ell \cong \Mat_2(\ZZ_\ell)$ acts by left multiplication. Then we have an isomorphism
  \begin{equation*}
    \OO \otimes \ZZ_\ell \cong \Mat_2(\ZZ_\ell)^{\op}
  \end{equation*}
  induced by right multiplication (to commute with the left action of $\OO_B$) by $\Mat_2(\ZZ_\ell)$ on $\Ta_\ell(A_0)$. We will identify $\Mat_2(\ZZ_\ell)^{\op} \cong \Mat_2(\ZZ_\ell)$ by taking the transpose $B \mapsto B^T$.
\end{cor}

\begin{cor}\label{isog thm false 2}
  Let $(A,\iota), (A_0,\iota_0)$ be supersingular false elliptic curves with canonical structure over $\fps$. For any $\ell \nmid p\delta$, fix isomorphisms $\Ta_\ell(A) \cong \Mat_2(\ZZ_\ell)$ and $\Ta_\ell(A_0) \cong \Mat_2(\ZZ_\ell)$ such that $\OO_B \otimes \ZZ_\ell \cong \Mat_2(\ZZ_\ell)$ acts on each by left multiplication. Then, as a $\ZZ_\ell$-module, we have an isomorphism
  \begin{equation*}
    \Hom_{\OO_B}(A,A_0) \otimes \ZZ_\ell \cong \Mat_2(\ZZ_\ell)^{\op}
  \end{equation*}
  given by right multiplication $\Ta_\ell(A) \cong \Mat_2(\ZZ_\ell) \to \Ta_\ell(A_0) \cong \Mat_2(\ZZ_\ell)$. In view of the previous Corollary, we instead identify $\Hom_{\OO_B}(A,A_0) \otimes \ZZ_\ell \cong \Mat_2(\ZZ_\ell)$ by composing the above isomorphism with the transpose map.
\end{cor}

The goal of this section is to produce a Deuring correspondence for supersingular false elliptic curves. We already know from \autoref{super false prop} that the endomorphism ring of a supersingular false elliptic curve is an order in the quaternion algebra $D$. We now show that it is a maximal order.

\begin{prop}
  The endomorphism ring $\OO = \End_{\OO_B}(A_0)$ of a supersingular false elliptic curve $(A_0,\iota_0)$ is a maximal order in the quaternion algebra $D$.
\end{prop}
\begin{proof}
  We may assume that $(A_0,\iota_0)$ is given by its canonical $\fps$-structure. It is equivalent to prove that $\End_{\OO_B}(A_0) \otimes \ZZ_\ell$ is a maximal order in $D \otimes \QQ_\ell$ for all finite primes $\ell$.

  We first consider the case where $\ell \neq p$, in which case \autoref{isog thm false} tells us that $\End(A_0) \otimes \ZZ_\ell \cong \Mat_4(\ZZ_\ell)$, and we fix this isomorphism.

  It is well known (see for example \cite[Lemma 7.5.4]{Voi1}) that if $C$ is a central simple $F$-algebra of dimension $d$, then $C \otimes C^{\op} \cong \End_{F}(C) \cong \Mat_d(F)$ via the isomorphism sending $\sum \alpha_i \otimes \beta_i \mapsto (\mu \mapsto \sum\alpha_i\mu\beta_i)$. Under this isomorphism, the transpose on $\End_F(C)$ sending $(\mu \mapsto \sum\alpha_i\mu\beta_i) \mapsto (\mu \mapsto \sum\beta_i\mu\alpha_i)$ induces an isomorphism of algebras $C \otimes 1 \leftrightarrow 1 \otimes C^{\op}$.

  When $\ell \neq p$, we can apply this to the quaternion algebra $B \otimes \QQ_\ell$ over $\QQ_\ell$, where $D \otimes \QQ_\ell \cong B \otimes \QQ_\ell$ by definition of the ramification in $D$, and $B \otimes \QQ_\ell \cong (B \otimes \QQ_\ell)^{\op}$ under the involution map (quaternionic conjugate). So we have an isomorphism $(B \otimes \QQ_\ell) \otimes (D \otimes \QQ_\ell) \cong M_4(\QQ_\ell)$ such that the transpose on $M_4(\QQ_\ell)$ induces an isomorphism $(B \otimes \QQ_\ell) \otimes 1 \cong 1 \otimes (D \otimes \QQ_\ell)$. Because the transpose on $M_4(\QQ_\ell)$ stabilises $M_4(\ZZ_\ell)$, we have an isomorphism
  \begin{equation*}
    \big((B \otimes \QQ_\ell) \otimes 1\big) \bigcap M_4(\ZZ_\ell)\quad \cong \quad \big(1 \otimes (D \otimes \QQ_\ell)\big) \bigcap M_4(\ZZ_\ell).
  \end{equation*}
  Consequently, if the left hand side is a maximal order in $B\otimes \QQ_\ell$ then the right hand side is a maximal order in $D \otimes \QQ_\ell$. Because $(A_0,\iota_0)$ is a false elliptic curve, we have an inclusion $\OO_B \otimes \ZZ_\ell \subset \End(A_0) \otimes \ZZ_\ell \cong \Mat_4(\ZZ_\ell)$, so that the left hand side is indeed a maximal order. The right hand side can be identified with
  \begin{equation*}
    \big(\End_{\OO_B}(A_0) \otimes \QQ_\ell\big) \bigcap \big( \End(A_0) \otimes \ZZ_\ell\big)  = \OO \otimes \ZZ_\ell.
  \end{equation*}
  Hence $\OO \otimes \ZZ_\ell$ is a maximal order in $D \otimes \QQ_\ell$ for $\ell \neq p$.

  At the prime $\ell = p$, we need to show that $\OO \otimes \ZZ_p$ is the valuation ring of the division algebra $D \otimes \QQ_p$. Here the valuation $v_p$ of an endomorphism of $(A_0,\iota_0)$ is defined to be half the $p$-adic valuation of its false degree (so the valuation of $[p]$ is 1) - one sees that this extends the usual $p$-adic valuation on $\QQ_p$ to $ D \otimes \QQ_p$ and hence is the unique such valuation. It is equivalent to show that $\OO_{(p)}$ consists exactly of the rational endomorphisms of $(A_0,\iota_0)$ whose false degree has nonnegative $p$-adic valuation. Indeed, we can write any rational endomorphism of $(A_0,\iota_0)$ as $a \phi$ for $a \in \QQ$ and $\phi \in \OO$, and moreover assume that $\phi$ is not an integer multiple of an endomorphism of $(A_0,\iota_0)$ (else change $a$ and $\phi$). By \cite[Lemma 9]{DT1}, $\phi$ has positive valuation $v_p(\phi) \geq r/2$ if and only if $\phi$ factors through the $r$-th power of the Frobenius morphism. Since the square of the Frobenius morphism is $[-p]$, our assumption on $\phi$ implies that $0 \leq v_p(\phi) \leq \frac{1}{2}$. Thus $v_p(a\phi) \geq 0$ implies $v_p(a) \geq -\frac{1}{2}$, and so $a \in \ZZ_p$. This completes the proof that $\OO$ is a maximal order at all primes, and so is a maximal order in $D$.
\end{proof}

\begin{remark}
  By \cite[Proposition 16.6.15]{Voi1}, any nonzero fractional ideal of a maximal order of a quaternion algebra is invertible.
\end{remark}

In order to establish a Deuring correspondence for supersingular false elliptic curves, we need the following definitions and results from \cite{Wat1} specialised to the case of false elliptic curves. Although the results in \cite{Wat1} are described in terms of abelian varieties over a finite field, we can apply them to the case of false elliptic curves (abelian surfaces with quaternionic multiplication) using the fact that if the underlying abelian surfaces of false elliptic curves are isogenous, then they are isogenous as false elliptic curves (see \cite[page 179]{Mil1}). Fix a supersingular false elliptic curve $(A_0,\iota_0)$ with canonical structure over $\fps$ and let $\OO= \End_{\OO_B}(A_0)$, a maximal order in the quaternion algebra $D$.

\begin{defn}
  Let $I$ be an invertible left $\OO$-ideal, and let $H(I)$ be the scheme-theoretic intersection
  \begin{equation*}
    H(I):= \bigcap\limits_{\alpha \in I} \ker\alpha,
  \end{equation*}
  which is a finite subgroup scheme of $A_0$ stable under $\OO_B$. We say that $I$ is a kernel ideal if we have the equality $I = \{ \alpha \in \OO \mid \alpha(H(I))=0\}$.
\end{defn}

\begin{thm}\label{wat1}
  Every invertible left $\OO$-ideal $I$ is a kernel ideal, and $\rank_{\OO_B} (H(I)) = \Nm(I)$ where $\rank_{\OO_B} (H(I))$ is the rank of $H(I)$ as an $\OO_B$-module. Alternatively we can write this as $\rank (H(I)) = \Nm(I)^2$.
\end{thm}
\begin{proof}
  This is \cite[Theorem 3.15]{Wat1}, where we use the fact that $\OO$ is a maximal order in $D$. The reason for the squaring of $\Nm(I)^2$ is that the false degree of an isogeny of false elliptic curves is the square root of the degree of the underlying isogeny of abelian varieties.
\end{proof}

\begin{thm}\label{wat2}
  Let $I$ and $J$ be kernel ideals. Then $A_0/H(I) \cong A_0/H(J)$ (as false elliptic curves) if and only if $I$ and $J$ are isomorphic as $\OO$-modules.
\end{thm}
\begin{proof}
  \cite[Theorem 3.11]{Wat1}.
\end{proof}

\begin{defn}
  For $H$ a finite subgroup scheme of $A_0$ stable under $\OO_B$, define
  \begin{equation*}
    I(H) := \{\alpha \in \OO \mid \alpha (H)=0\},
  \end{equation*}
  a left $\OO$-ideal, nonzero (and hence invertible) since $\#H \in I(H)$.
\end{defn}
So $I$ being a kernel ideal says that $I(H(I))=I$, justifying our notation.

\begin{lemma}\label{HI inverse}
  If $H$ is any finite subgroup scheme of $A_0$ stable under $\OO_B$, then $H(I(H))=H$.
\end{lemma}
\begin{proof}
  By the connected-\'{e}tale exact sequence, $H$ factors as $H^0 \times H^{\text{\'{e}t}}$ where $H^0$ is the connected component of the identity and $H^{\text{\'{e}t}}$ is an \'{e}tale finite group scheme. The subgroup scheme $H^0$ can be interpreted as the kernel of a certain power of the absolute Frobenius morphism. Since the Frobenius commutes with $\OO_B$, the decomposition is compatible with the action of $\OO_B$, and we have $H(I(H)) = H(I(H^0)) \times H(I(H^{\text{\'{e}t}}))$. So we may assume that $H$ is either the kernel of the $r$-th power of the Frobenius for some $r$, or is \'{e}tale. In the first case, the result is obvious as $I(H)$ is the ideal generated by the $r$-th power of the Frobenius. In the case that $H$ is \'{e}tale, if $H$ is properly contained in $H(I(H))$ then the same is true of the base changes to $\fpbar$, which are constant group schemes. We can identify these group schemes with subgroups of the group $A_0[N]$ where $N$ is the rank of $H$. Note that $N$ is coprime to $p$ because $A_0[p]$ is the kernel of the square of the Frobenius by \autoref{super false prop}. By \autoref{isog thm false}, we have the isomorphism $\OO/N\OO \cong \End_{\OO_B}(A_0[N])$. Hence if $H$ is properly contained in $H(I(H))$ as subgroups of $A_0[N]$, then there exists $\alpha \in \OO$ annihilating $H$ but not $H(I(H))$. This contradicts the fact that $I(H)=I(H(I(H)))$, coming from applying \autoref{wat1} to the ideal $I=I(H)$.
\end{proof}

\begin{thm}\label{false deuring}
  Fix a supersingular false elliptic curve $(A_0,\iota_0)$ with respect to $B$ and let $\OO = \End_{\OO_B}(A_0)$, a maximal order in the quaternion algebra $D$. There is a bijection between the set of isomorphism classes of supersingular false elliptic curves and the set of isomorphism classes of invertible left $\OO$-ideals (the left class set $\Cls_L\OO$). The map is given by sending $(A,\iota)$ to the isomorphism class of the ideal $\Hom_{\OO_B}(A,A_0)\theta$, for any isogeny $\theta\st (A_0,\iota_0) \to (A,\iota)$.
\end{thm}
\begin{proof}
  We construct the inverse by sending the isomorphism class of an invertible left $\OO$-ideal $I$ to the isomorphism class of the supersingular false elliptic curve $(A_0/H(I),\iota_0)$. By \autoref{wat1} and \autoref{wat2}, this is a well-defined injection. It remains to show that any supersingular false elliptic curve $(A,\iota)$ is isomorphic to $(A_0/H(J),\iota_0)$ for some invertible left $\OO$-ideal $J$. By \autoref{super false prop}, for any supersingular $(A,\iota)$, there exists an isogeny $\theta\st A_0 \to A$, which we may assume is separable after factoring out a power of the Frobenius. Then $(A,\iota)$ is isomorphic to $(A_0/\ker\theta,\iota_0)$. By \autoref{HI inverse}, $\ker\theta = H(I(\ker\theta))$ so we can take $J=I(\ker\theta)$. Since $\theta$ is separable, we in fact have $I(\ker\theta) = \Hom_{\OO_B}(A,A_0)\theta$, whose isomorphism class as an $\OO$-ideal does not depend on $\theta$, so that the bijection we have defined relates $(A,\iota)$ to $J=\Hom_{\OO_B}(A,A_0)\theta$.
\end{proof}

Finally, we must enrich this correspondence with the differential and level structure on the supersingular false elliptic curves used to define modular forms on Shimura curves.

\begin{lemma}\label{insep}
  Suppose $\phi\st A \to A_0$ is an isogeny of supersingular false elliptic curves, with $\OO= \End_{\OO_B}(A_0)$, and let $\omega$ be a basis for $\om_{A_0}$. Then $(\phi^t)^*\omega=0$ if and only if $\phi \in \pi \Hom_{\OO_B}(A,A_0)$.
\end{lemma}
\begin{proof}
  Let $\theta\st A_0 \to A$ be any separable isogeny. Then $0 \neq \phi\theta \in \OO = \End_{\OO_B}(A_0)$. By definition of $\pi$ as a uniformiser of $\OO_p$, we see that $x \in \pi\OO$ if and only if the false degree of $x$ as an endomorphism of $A_0$ is divisible by $p$. Now $(\phi^t)^*\omega = 0$ if and only if the false degree of $\phi$ is divisible by $p$, since $A_0$ is supersingular. This occurs if and only if $\phi\theta$ has false degree divisible by $p$ (since $\theta$ has false degree coprime to $p$), and we have seen that this is equivalent to $\phi\theta \in \pi\OO$. The statement $\phi\theta \in \pi\OO$ implies $\phi \in \pi \Hom_{\OO_B}(A,A_0)$ by precomposing with $\theta^t$, and the converse is immediate.
  % square of Frob is -p so if false degree div by p then factors through frob
\end{proof}

\begin{lemma}\label{crt}[Chinese remainder theorem for ideals of $\OO$]
  Let $I$ be an invertible left $\OO$-module. The natural map $I/\pi N I \to I/\pi I \times I/NI$ is an isomorphism.
\end{lemma}
\begin{proof}
  The map is injective because $p$ and $N$ are coprime, and comparing the size of each side, it is equivalent to show that
  \begin{equation*}
    \# (\OO /\pi N \OO) = \# (\OO / \pi \OO) \times  \#(\OO / N\OO).
  \end{equation*}
  But this just follows from the fact that for any $a \in \OO$, $\# (\OO / a \OO)$ is the norm (not reduced norm) $N_{D/\QQ}(a)$ (\cite[Lemma 9.6.3]{Voi1}) and that the norm is multiplicative. %det of mult is norm (not reduced norm)
\end{proof}

\begin{defn}\label{quadruple def}
  Fix level structure $V = V(N)$ or $V_1(N)$. By a supersingular false elliptic quadruple of level $V$, we mean a quadruple $(A,\iota,\alpha,\omega)$ consisting of a supersingular false elliptic curve $(A,\iota)$ with canonical $\fps$-structure, a basis $\omega$ for $\om_{A}$ rational over $\fps$, and level $V$-structure $\alpha$ (defined over $\fpbar$).
\end{defn}

\begin{thm}\label{enrich false deuring}
  Fix a supersingular false elliptic quadruple $(A_0,\iota_0,\alpha_0,\omega_0)$ of level $V$. For any supersingular false elliptic quadruple $(A,\iota,\alpha,\omega)$, there exists an isogeny $\phi\st (A,\iota) \to (A_0,\iota_0)$ such that $(\phi^t)^*\omega = \omega_0$ and $\alpha_\phi = \alpha_0$. Let $\OO = \End_{\OO_B}(A_0)$, a maximal order in $D$, so that $\Hom_{\OO_B}(A,A_0)$ is an invertible left $\OO$-module. We have the following uniqueness results:
  \begin{itemize}
    \item Over all such isogenies $\phi$, the reduction of $\phi \mod \pi \Hom_{\OO_B}(A,A_0)$ is unique.
    \item For $V=V(N)$, the reduction of $\phi \mod N\Hom_{\OO_B}(A,A_0)$ is unique.
    \item For $V=V_1(N)$, for some isomorphism $(\OO/N\OO)^\times \cong \GL_2(\ZZ/N \ZZ)$, the reduction of $\phi \mod N\Hom_{\OO_B}(A,A_0)$ is unique up to left multiplication by the subgroup of $\GL_2(\ZZ/N\ZZ)$ consisting of matrices of the form $\begin{psmallmatrix} 1&* \\0&*\end{psmallmatrix}$.
  \end{itemize}
  %Any such $\phi$ is a basis for $\mathrm{Hom}_{\OO_B}(A,A_0)/\pi N \mathrm{Hom}_{\OO_B}(A,A_0)$ as an $\OO/\pi N \OO$-module.
\end{thm}
\begin{proof}
  By \autoref{crt}, if we can construct isogenies $\phi_\pi, \phi_N\st (A,\iota) \to (A_0,\iota_0)$ such that
  \begin{equation*}
    (\phi_\pi^t)^*\omega = \omega_0\quad\text{and}\quad \alpha_{\phi_N} = \alpha_0,
  \end{equation*}
  then there exists an isogeny $\phi\st (A,\iota) \to (A_0,\iota_0)$ congruent to $\phi_\pi \mod \pi \Hom_{\OO_B}(A,A_0)$, and congruent to $\phi_N \mod N\Hom_{\OO_B}(A,A_0)$, which then satisfies the desired properties.

  We begin with the construction of $\phi_\pi$. Start with any separable isogeny $\psi\st A \to A_0$. Then $(\psi^t)^* \omega$ is a basis for $\om_{A_0}$, rational over $\fps$ (since all isogenies between supersingular false elliptic curves with canonical structure are defined over $\fps$). By composing with an endomorphism of $A_0$, it suffices to prove that $\OO$ acts transitively by pullback on the set of nonzero elements of $\om_{A_0}$ rational over $\fps$ (which can be identified with $\fps^\times$ by picking a basis). By \autoref{insep}, $(\OO /\pi \OO)^\times \cong (\OO_p/\pi \OO_p)^\times \cong \fps^\times$ acts faithfully on this set, and so must act transitively. Moreover, $\phi_\pi$ mod $\pi \Hom_{\OO_B}(A,A_0)$ is unique and nonzero by \autoref{insep}.

  Let $H$ be $\OO_B \otimes \ZZ/N\ZZ \cong \Mat_2(\ZZ/N\ZZ)$ or $\ZZ/N\ZZ \times \ZZ/N\ZZ$ as appropriate. We want the existence of an isogeny $\phi_N\st (A,\iota) \to (A_0,\iota_0)$ such that the following diagram commutes:
  % $$\xymatrix{
  %   & H \ar[ld]_\alpha \ar[rd]^{\alpha_0} & \\
  %   A[N] \ar[rr]^{\phi_N}_\sim & & A_0[N].
  %   }$$
  \begin{center}
    \begin{tikzcd}
      & H \arrow[ld, "\alpha"'] \arrow[rd, "\alpha_0"] & \\
      A[N] \arrow[rr, "\phi_N", "\sim"'] & & A_0[N].
    \end{tikzcd}
  \end{center}

  Since $\alpha,\alpha_0$ are inclusions of $\OO_B$-modules, and $A[N],A_0[N]$ are isomorphic to $\Mat_2(\ZZ/N\ZZ)$ over $\fpbar$ as $\OO_B$-modules, there always exists an isomorphism of $\OO_B$-modules $A[N] \cong A_0[N]$ making the diagram commute. This can then be viewed as an element of $\GL_2(\ZZ/N\ZZ)$ acting on $\Mat_2(\ZZ/N\ZZ)$ by right multiplication, and by \autoref{isog thm false 2} this must come from $\Hom_{\OO_B}(A,A_0) \otimes \ZZ/N\ZZ$. Hence such an isogeny $\phi_N$ exists.

  In the case that $H=\Mat_2(\ZZ/N\ZZ)$, it is clear that $\phi_N$ is unique mod $N$ as there is a unique element of $\GL_2(\ZZ/N\ZZ)$ making the diagram commute.

  Now consider $H=\ZZ/N\ZZ \times \ZZ/N\ZZ$. Under an automorphism of $A_0[N]$, we may assume that $\alpha_0$ maps $H$ to the subgroup $\begin{psmallmatrix}*&0\\ *&0\end{psmallmatrix}$ of $A_0[N] \cong \Mat_2(\ZZ/N\ZZ)$. This subgroup is generated as an $\OO_B$-module by $\begin{psmallmatrix}1&0\\0&0\end{psmallmatrix}$. Then $\phi_N \mod N\Hom_{\OO_B}(A,A_0)$ is unique up to postcomposition by an element of $\OO \otimes \ZZ/N\ZZ$ whose induced action on $A_0[N]$ fixes $\begin{psmallmatrix}1&0\\0&0\end{psmallmatrix}$. These are given by right multiplication by matrices congruent to $\begin{psmallmatrix}1&0\\ *&*\end{psmallmatrix}$ mod $N$, and in $\OO \otimes \ZZ /N\ZZ \cong \Mat_2(\ZZ/N\ZZ)$ these are given by matrices congruent to $\begin{psmallmatrix}1&*\\0&*\end{psmallmatrix}$ mod $N$ (recall we must take transposes).
\end{proof}

\begin{thm}\label{main thm 2}
  For $V = V(N),V_1(N)$, let $\Omega=\Omega^D(N),\Omega^D_1(N)$, respectively.
  There exists a natural bijection between the set of isomorphism classes of supersingular false elliptic quadruples of level $V$, and $\Omega$.
\end{thm}

\begin{proof}
  Fix a supersingular false elliptic quadruple $(A_0,\iota_0,\alpha_0,\omega_0)$ of level $V$ and let $\OO=\End_{\OO_B}(A_0)$. For any supersingular false elliptic quadruple $(A,\iota,\alpha,\omega)$, fix some separable isogeny $\theta\st A_0 \to A$. Let $\phi\st A \to A_0$ be an isogeny as in the statement of \autoref{enrich false deuring}. Then $I=\Hom_{\OO_B}(A,A_0)\theta$ is an invertible left $\OO$-ideal containing the element $\phi\theta \in I$, whose reduction modulo $\pi N$ is a basis for $I/\pi N I$ as an $\OO/\pi N \OO$-module. Since $I$ is invertible, it is locally principal (\cite[Theorem 16.6.1]{Voi1}), and we can pick local generators $x_\ell \in D^\times_\ell$ so that $I \otimes \ZZ_\ell = \OO_\ell x_\ell$ for all $\ell$, and also $\phi\theta \equiv x_\ell \mod \pi N (I \otimes \ZZ_\ell)$ for all $\ell$ (note that almost all of these congruences are vacuous). The class $[x_\ell]$ of $(x_\ell)$ in $\Omega$ depends only on the isomorphism class of the supersingular false elliptic quadruple $(A,\iota,\alpha,\omega)$ by the uniqueness properties of $\phi$ in \autoref{enrich false deuring}. So we have a well defined map between the set of supersingular quadruples of level $V$ and the set $\Omega$.

  We construct an inverse to this map. For any $[x_\ell] \in \Omega$, the $(x_\ell) \in D^\times (\af)$ are local generators for an invertible left fractional $\OO$-ideal $J$. We can pick a representative $I$ of the left class set of $J$ that is an $\OO$-ideal of reduced norm coprime to $pN$. Replacing $J$ by $I$ corresponds to multiplication of $(x_\ell)$ on the right by an element of $D^\times(\QQ)$. So assume that $(x_\ell)$ are local generators for $I$. There also exists some $\psi \in I$ congruent to each $x_\ell$ mod $\pi N I_\ell$ by \autoref{crt}. From \autoref{false deuring}, we have the supersingular false elliptic curve $A = A_0/H(I)$ and a separable isogeny $\theta\st A_0 \to A$ of false degree $\Nm(I)$ (\autoref{wat1}) coprime to $pN$, such that $\Hom_{\OO_B}(A,A_0)\theta = I$. Hence we can express $\psi\in I$ as $\phi\theta$ for some isogeny $\phi\st A \to A_0$. The desired inverse map is given by sending $[x_\ell] \in \Omega$ to the isomorphism class of the supersingular false elliptic quadruple $(A,\iota,\alpha,\omega)$, for $\alpha, \omega$ some differential and level structure satisfying $(\phi^t)^*\omega = \omega_0$ and $(\alpha)_\phi = \alpha_0$. Such $\alpha$ and $\omega$ exist: we can take $\omega = \frac{1}{\mathrm{falsedeg}\phi} \phi^* \omega_0 \neq 0$ and $\alpha = (\alpha_0)_{\phi^{-1}}$, where $\phi^{-1}$ denotes an inverse to $\phi$ mod $N$ ($\mathrm{falsedeg}(\phi)$ is coprime to $N$ since $\psi$ is a basis for $I/\pi N I$ and the reduced norm of $I$ is coprime to $pN$; consequently $\phi$ induces an isomorphism $A[N] \cong A_0[N]$). This proves the theorem.
\end{proof}

Fix a supersingular quadruple $(A_0,\iota_0,\alpha_0,\omega_0)$ and let a supersingular quadruple $(A,\iota,\alpha,\omega)$ correspond to $[x_\ell] \in \Omega$. Recall that weight $k$ modular forms with respect to $B^\times$ are supposed to satisfy an automorphy condition $f(A,\iota,\alpha,\mu\omega)=\mu^{-k}f(A,\iota,\alpha,\omega)$ for all $\mu \in \fps^\times$. We want to determine the image of $(A,\iota,\alpha,\mu\omega)$ in $\Omega$ in terms of $[x_\ell]$ and $\mu$. For an isogeny $\phi\st A \to A_0$ as in \autoref{enrich false deuring} for the quadruple $(A,\iota,\alpha,\omega)$, one only needs to modify $\phi$ mod $\pi \Hom_{\OO_B}(A,A_0)$ to produce the relevant isogeny $A\to A_0$ for the triple $(A,\iota,\alpha,\mu\omega)$. From the proof of \autoref{main thm 2}, we see that we must only modify $(x_\ell)$ at the $p$-th place to obtain the element of $\Omega$ corresponding to $(A,\iota,\alpha,\mu\omega)$. The $x_\ell$ are still local generators for the same ideal $\Hom_{\OO_B}(A,A_0)\theta$, and so we find that $(A,\iota,\alpha,\mu\omega)$ corresponds to $[y_\ell] \in \Omega$ where $y_\ell = x_\ell$ for $\ell \neq p$, and $y_p = \mu' \cdot x_p$ for some $\mu' \in \OO_p^\times$, where $[y_\ell]$ only depends on the class of $\mu'$ in $\OO_p^\times/\OO_p^\times(\pi) \cong \fps^\times$.

\begin{notn}
  Let $Q_k$ denote the space of rules $f$ assigning a value $f(A,\iota,\alpha,\omega) \in \fpbar$ to each isomorphism class of quadruples $(A,\iota,\alpha,\omega)$ where $(A,\iota,\alpha)$ is a supersingular false elliptic curve over $\fps$ with level $V$ structure, and $\omega$ is a basis for $\om_{A/\fps}$, satisfying
  \begin{equation*}
    f(A,\iota,\alpha,\mu\omega) = \mu^{-k}f(A,\iota,\alpha,\omega) \text{ for any $\mu \in \fps^\times$.}
  \end{equation*}
\end{notn}

\begin{cor}\label{weight action}
  The bijection of \autoref{main thm 2} induces a bijection between $Q_k$ and $\MM_k(\Omega)$.
\end{cor}

% for the supersingular false elliptic curves we worked just with F_p^2*. On D we fixed an isomorphism O_p^*/O_p*(1) \cong F_p^2*. We can choose this isomorphism to make everything work.

We can also interpret $Q_k$ as follows. Let $\cQ_k$ fit into the short exact sequence
% $$\xymatrix{0 \ar[r] & \underline{\omega}_{\XX}^{\otimes k-(p-1)} \ar[r]^{\HH} & \underline{\omega}_{\XX}^{\otimes k} \ar[r] & \cQ_k^B \ar[r] & 0.}$$
\begin{center}
  \begin{tikzcd}
    0 \arrow[r] & \om^{\otimes k-(p-1)} \arrow[r, "\HH"] & \om^{\otimes k} \arrow[r] & \cQ_k \arrow[r] & 0.
  \end{tikzcd}
\end{center}

Since the Hasse invariant $\HH$ vanishes exactly at the supersingular false elliptic curves, we see that $Q_k = H^0(\XX, \cQ_k)$.

Taking the long exact sequence of cohomology we have
% \begin{equation}\label{eqn:les}
%   \xymatrix{0 \ar[r] & M^B_{k-(p-1)} \ar[r]^{\HH} & M^B_k \ar[r]& Q_k^B \ar[r] & H^1(\XX, \underline{\omega}_{\XX}^{\otimes k-(p-1)}) \ar[r] &\dots}
% \end{equation}
\begin{equation}\label{eqn:les}
  \begin{tikzcd}
    0 \arrow[r] & M_{k-(p-1)} \arrow[r, "\HH"] & M_k \arrow[r] & Q_k \arrow[r] & H^1(\XX, \om^{\otimes k-(p-1)}) \arrow[r] & \dots
  \end{tikzcd}
\end{equation}

\begin{lemma}
  For $k \geq 3$, the cohomology group $H^1(\XX,\om^{\otimes k})$ vanishes.
\end{lemma}
\begin{proof}
  Applying Serre duality to the Shimura curve $\XX$, together with the Kodaira--Spencer isomorphism, \autoref{KS2}, we can identify this cohomology group with $H^0(\XX, \om^{\otimes 2-k})^\vee$. The degree of the line bundle $\om$ is positive (the genus of the Shimura curve is strictly greater than $1$, see \cite[Lemma 6]{DT1}). Then for $k \geq 3$, $H^0(\XX,\om^{\otimes 2-k})=0$, as a line bundle of negative degree has no nonzero global sections.
\end{proof}

\begin{cor}\label{wkskB}
  For $k > p+1$, we have $\MM_k(\Omega) \cong Q_k \cong M_k/\HH M_{k-(p-1)}$ as $\fpbar$-vector spaces.
\end{cor}

\section{Hecke eigenvalues}

We adapt Serre's argument to show that the systems of Hecke eigenvalues arising from the spaces $M_k$ over all $k$ are the same as those arising from the spaces $\MM_k(\Omega)$ over all $k$. There are two steps to this. The Hecke operators are defined on $Q_k$ by restricting the Hecke operators on the Shimura curve $\XX$ to the supersingular locus, which is Hecke invariant since it is closed under prime-to-$p$ isogenies.
Following \autoref{weight action}, we must show that the bijection $Q_k \cong \MM_k(\Omega)$ is Hecke equivariant, and we must also show that the Hecke eigenvalues arising from the $M_k$ over all $k$ are the same as those arising from the $Q_k$ over all $k$. We point out that our work, as in \cite{serre-letters}, relates systems of Hecke eigenvalues arising from $M_k$ between \emph{explicit} weights $k$, allowing for a refined understanding of the weights for which a system of Hecke eigenvalues appears.

% Recall that the Hecke operators on $\MM_k(\Omega)$ were defined as follows:
% \begin{equation*}
%   T_\ell f (x) = \frac{1}{\ell} \sum_g f\big(\iota_\ell(g)x\big),\qquad\text{where}\qquad
%   \GL_2(\ZZ_\ell)\begin{pmatrix}1 & 0 \\ 0 & \ell\end{pmatrix}\GL_2(\ZZ_\ell)=\bigsqcup_g \GL_2(\ZZ_\ell)g.
% \end{equation*}

\begin{thm}\label{thm:hecke_equiv}
  The bijection $Q_k \cong \MM_k(\Omega)$ from \autoref{weight action} is Hecke-equivariant.
\end{thm}
\begin{proof}
  We show equivariance for an arbitrary Hecke operator $T_{\ell_0}$.

  We recall how the bijection of \autoref{weight action} and \autoref{main thm 2} was constructed. Fix a supersingular false elliptic quadruple $(A_0,\iota_0,\alpha_0,\omega_0)$ of level $V$, and consider another quadruple $(A,\iota,\alpha,\omega)$. Fix an isogeny $\theta\st (A_0,\iota_0) \to (A,\iota)$. Then the image $[x_\ell] \in \Omega$ of $(A,\iota,\alpha,\omega)$ is determined by the conditions:
  \begin{enumerate}[label=(\roman*)]
    \item $\Hom_{\OO_B}(A,A_0) \theta \otimes \ZZ_\ell = \OO_\ell x_\ell$.
    \item If $\phi \in \Hom_{\OO_B}(A,A_0)$ satisfies $(\phi^t)^* \omega = \omega_0$ and $\alpha_\phi = \alpha_0$, then for all $\ell$,
      \begin{equation*}
        \phi \theta \equiv x_\ell \mod \pi N \OO_\ell x_\ell.
      \end{equation*}
  \end{enumerate}
  Let $\eta\st A \to A/G$ be an isogeny with $G \cong \ZZ/\ell_0\ZZ \times \ZZ/\ell_0\ZZ$ as $\oh_B \otimes \ZZ/\ell_0\ZZ$-modules. We want to determine the image of $(A/G,\iota,\alpha_\eta, (\eta^t)^*\omega)$ in $\Omega$.
  % $$\xymatrix{ A/G \ar@<1ex>[rr]^{\eta^t} && A \ar@<1ex>[ll]^{\eta} \ar@<1ex>[rr]^\phi && A_0 \ar@<1ex>[ll]^\theta}$$
  \begin{center}
    \begin{tikzcd}
      A/G \arrow[rr, shift left=1ex, "\eta^t"] && A \arrow[ll, shift left=1ex, "\eta"] \arrow[rr, shift left=1ex, "\phi"] && A_0 \arrow[ll, shift left=1ex, "\theta"]
    \end{tikzcd}
  \end{center}
  Now $\phi \eta^t \in \Hom_{\OO_B}(A/G,A_0)$ satisfies $((\phi \eta^t)^t)^* (\eta^t)^* \omega = \ell_0 \omega_0$ and $\alpha_{\phi \eta^t \eta} = \ell_0 \alpha_0$. Let $m$ be an integer that is an inverse of $\ell_0$ mod $p N$. Then $\phi_G := m \phi \eta^t \in \Hom_{\OO_B}(A/G,A_0)$ satisfies $(\phi_G^t)^* (\eta^t)^*\omega = \omega_0$ and $\alpha_{\phi_G \eta} = \alpha_0$. Hence the image $[y_\ell] \in \Omega$ of $(A/G,\iota,\alpha_\eta, (\eta^t)^*\omega)$ is determined by the conditions:
  \begin{enumerate}[label=(\roman*)]
    \item $\Hom_{\OO_B}(A/G,A_0)\eta \theta \otimes \ZZ_\ell = \OO_\ell y_\ell$.
    \item $\phi_G \eta\theta \equiv \phi\theta \equiv y_\ell \mod \pi N \OO_\ell y_\ell.$
  \end{enumerate}
  Because $\eta$ has false degree $\ell_0$, we see that for any $\ell \neq \ell_0$,
  \begin{equation*}
    \Hom_{\OO_B}(A/G,A_0)\eta\theta \otimes \ZZ_\ell = \Hom_{\OO_B}(A,A_0)\theta \otimes \ZZ_\ell.
  \end{equation*}
  Consequently, we can take $y_\ell$ to be $x_\ell$ for all $\ell \neq \ell_0$. At the $\ell_0$ place we only need to consider condition (i) since $\ell_0 \nmid pN$.

  Fixing a basis for $\Ta_{\ell_0}(A)$, for each $\eta\st A \to A/G$ we can find a basis for $\Ta_{\ell_0}(A/G)$ such that the linear map $\Ta_{\ell_0}(A) \cong \Mat_2(\ZZ_{\ell_0}) \to \Ta_{\ell_0}(A/G) \cong \Mat_2(\ZZ_{\ell_0})$ induced by $\eta$ is given by right multiplication by a matrix $g_G^T$ of the form $\begin{psmallmatrix}1&0\\i&\ell_0\end{psmallmatrix}$ or $\begin{psmallmatrix} \ell_0 &0 \\0 &1 \end{psmallmatrix}$. Each matrix of this form occurs exactly once as we range over all $\eta$ so the $g_G$ are the coset representatives of $\GL_2(\ZZ_{\ell_0}) \backslash \GL_2(\ZZ_{\ell_0}) \begin{psmallmatrix} 1&0 \\0 & \ell_0 \end{psmallmatrix} \GL_2(\ZZ_{\ell_0})$. Recall that we identify $\Hom_{\OO_B}(A,A_0) \otimes \ZZ_{\ell_0} \cong \Mat_2(\ZZ_{\ell_0})$ by taking the transpose of the matrices acting by right multiplication. Hence, with respect to our choice of basis for the Tate module $\Ta_{\ell_0}(A)$, we see that (for any basis of $\Ta_{\ell_0}(A/G)$)
  \begin{equation*}
    \Hom_{\OO_B}(A/G,A)\eta \otimes \ZZ_{\ell_0} = \Mat_2(\ZZ_{\ell_0}) g_G.
  \end{equation*}

  Since $x_{\ell_0}$ was defined by $\Hom_{\OO_B}(A,A_0)\theta \otimes \ZZ_{\ell_0} = \OO_{\ell_0}x_{\ell_0}$, we see that $y_{\ell_0}$ can be given by $g_G \cdot x_{\ell_0}$ for $g_G$ corresponding to $G$. This proves the theorem.
\end{proof}

We now concern ourselves with the quotients $W_k = M_k/\HH M_{k-(p-1)}$. The Hecke operators are well defined on the quotient because $M_{k-(p-1)}$ is stable under them.

\begin{prop}\label{arise quotient 2}
  The systems of Hecke eigenvalues arising from the $M_k$ as $k$ ranges over the nonnegative integers, are the same as those arising from all the $W_k$.
\end{prop}
\begin{proof}
  If $f \in M_k$ is a Hecke eigenform with eigenvalues $(a_\ell)$, then the image $\bar{f}$ of $f$ in the quotient $W_k$ also satisfies $T_\ell(\bar{f}) = a_\ell \bar{f}$ for all $\ell \nmid p\delta N$. To show that the system of eigenvalues $(a_\ell)$ arises from one of the quotients $W_{k'}$, we need to show that there exists a Hecke eigenform $g \in M_{k'}$, also with eigenvalues $(a_\ell)$, such that $\bar{g} \in W_{k'}$ is nonzero. Let $r \geq 0$ be the largest degree to which $f$ vanishes at all the supersingular points on the Shimura curve $\XX$. Then $f$ factors as $f= \HH^r \cdot g$ for some $g \in M_{k-r(p-1)}$, where now $g$ is not a multiple of the Hasse invariant, so has nonzero image in $W_{k-r(p-1)}$. By \autoref{Hasse quotient 2}, $g$ is a Hecke eigenform with the same Hecke eigenvalues $(a_\ell)$. This implies that any system of Hecke eigenvalues $(a_\ell)$ arising from an $M_k$ also arises from some $W_{k'}$, where moreover we can choose $k' \equiv k \mod p-1$.

  Conversely, if $(a_\ell)$ is a system of Hecke eigenvalues arising from a quotient $W_k$, then \cite[Proposition 1.2.2]{AS1} tells us (since we are working over the algebraically closed field $\fpbar$) that $(a_\ell)$ also arises from $M_k$. We remark that while the system of eigenvalues `lifts' from $W_k$ to $M_k$, the eigenvectors need not all lift.
\end{proof}

\begin{lemma}\label{mod weight 2}
  There is a canonical Hecke equivariant isomorphism $Q_k \cong Q_{k+(p^2-1)}$ for any $k$.
\end{lemma}
\begin{proof}
  Let $f \in Q_k$. Then
  \begin{equation*}
    f(A,\iota,\alpha,\mu\omega)=\mu^{-k}f(A,\iota,\alpha,\omega) = \mu^{-(k+(p^2-1))}f(A,\iota,\alpha,\omega)
  \end{equation*}
  for any supersingular quadruple $(A,\iota,\alpha,\omega)$ and $\mu \in \fps^\times$. So $f$ defines an element of $Q_{k+(p^2-1)}$.
\end{proof}

\begin{prop}\label{Heckesuper}
  The systems of Hecke eigenvalues arising from the $M_k$ as $k$ ranges over the nonnegative integers, are the same as those arising from all the $Q_k$.
\end{prop}
\begin{proof}
  Suppose $(a_\ell)$ arises from some $M_k$. By \autoref{arise quotient 2}, $(a_\ell)$ arises from some $W_{k'}$. From the earlier long exact sequence \eqref{eqn:les}, we always have a Hecke equivariant injection $W_{k'} \hookrightarrow Q_{k'}$ (an isomorphism when $k' > p+1$), so that $(a_\ell)$ arises from $Q_{k'}$. Conversely, if $(a_\ell)$ arises from some $Q_k$, we can replace $k$ by $k+(p^2-1)$ by the previous lemma to assume that $k > p+1$. Then \autoref{wkskB} tells us that $(a_\ell)$ arises from $W_k$, and so by \autoref{arise quotient 2} arises also from $M_k$.
\end{proof}

% \begin{cor}
%   The systems of Hecke eigenvalues arising from the modular forms mod $p$, with respect to $B^\times$ of level $V$, all arise from weight $k \leq p^2-1$.
% \end{cor}

Putting together \autoref{weight action}, \autoref{eqn:les}, \autoref{wkskB}, and \autoref{thm:hecke_equiv} we get
\begin{thma}
  For every $k\geq 0$, restriction to the supersingular locus induces a Hecke-equivariant injective map $W_k\to\MM_k(\Omega)$.
  If $k>p+1$, this is an isomorphism.
\end{thma}

Finally, from Theorem A, Theorem B, and \autoref{arise quotient 2}, we have
\begin{corc}
  Let $(a_{\ell})\in\ds\prod_{\ell\nmid p\delta N} \fpbar$ be a system of Hecke eigenvalues.
  \begin{enumerate}
    \item If $(a_\ell)$ arises from $M_k$, then it arises from $\MM_{k'}(\Omega)$ for some $0\leq k'\leq k$ such that $k'\equiv k\pmod{p-1}$.
    \item If $(a_\ell)$ arises from $\MM_k(\Omega)$, then it arises from $M_{k'}$ for every $k'>p+1$ such that $k'\equiv k,pk\pmod{p^2-1}$.
  \end{enumerate}
  In particular, the systems of Hecke eigenvalues arising from $\ds\bigoplus_k M_k$ coincide with those arising from $\MM(\Omega)=\ds\bigoplus_{k\text{ mod }p^2-1} \MM_k(\Omega)$.
\end{corc}

% Combining \autoref{wkskB}, \autoref{thm:hecke_equiv}, and \autoref{Heckesuper}, we have the following:

% \begin{thm}\label{cor:main}
%   The systems of Hecke eigenvalues $(a_\ell)$ arising from the modular forms mod $p$ with respect to $B^\times$, over all weights $k$ and fixed level $V$, are the same as those arising from $\MM(\Omega)$.
% \end{thm}

\bibliographystyle{alpha}
\bibliography{refs}

\end{document}